\documentclass[12pt]{amsart}
\usepackage[latin1]{inputenc}
\usepackage{amsbsy}
\usepackage{amscd} 
\usepackage{amsfonts}
\usepackage{amsgen} 
\usepackage{amsmath}
\usepackage{amsopn} 
\usepackage{amssymb}
\usepackage{amstext}
\usepackage{amsthm} 
\usepackage{amsxtra}
\usepackage{rotating}
\usepackage{tikz}
\usepackage{pdflscape}
\usepackage{afterpage}
\usepackage{capt-of}

\theoremstyle{plain} 
\newtheorem{thm}{Theorem}[section]
\newtheorem{prop}[thm]{Proposition}
\newtheorem{lem}[thm]{Lemma}
\newtheorem{cor}[thm]{Corollary}
\theoremstyle{definition}
\newtheorem{defn}[thm]{Definition}
\newtheorem{rem}[thm]{Remark}

\numberwithin{equation}{section}

\theoremstyle{plain}
\newtheorem{innercustomthm}{Theorem}
\newenvironment{customthm}[1]
  {\renewcommand\theinnercustomthm{#1}\innercustomthm}
  {\endinnercustomthm}

\renewcommand{\theta}{\vartheta}
\renewcommand{\phi}{\varphi}
\renewcommand{\epsilon}{\varepsilon}
\renewcommand{\subset}{\subseteq}
\renewcommand{\supset}{\supseteq}

\newcommand{\N}{\mathbb N}
\newcommand{\Z}{\mathbb Z}

\newcommand{\C}{\mathbb C}

\newcommand{\CC}{\mathcal C}

\newcommand{\OOO}{\mathcal O}
\newcommand{\HHH}{\mathcal H}
\newcommand{\BBB}{\mathcal B}
\newcommand{\SSS}{\mathcal S}
\newcommand{\Rel}{\mathcal R}

\newcommand{\la}{\mathsf{a}}
\def\symstrut#1 {\vrule height #1 depth #1 width 0cm}

\usepackage{calc}
\newcounter{PartitionDepth}
\newcounter{PartitionLength}

\newcommand{\partii}[3]{
 \begin{picture}(#3,#1)
 \setcounter{PartitionLength}{#3-#2}
 \setcounter{PartitionDepth}{-1-#1}
 \put(#2,\thePartitionDepth){\line(0,1){#1}}     
 \put(#3,\thePartitionDepth){\line(0,1){#1}}
 \put(#2,\thePartitionDepth){\line(1,0){\thePartitionLength}}
 \end{picture}}

\newcommand{\upparti}[2]{
 \begin{picture}(#2,#1)
 \setcounter{PartitionDepth}{#1}
 \put(#2,0){\line(0,1){#1}}
 \end{picture}}
\newcommand{\uppartii}[3]{
 \begin{picture}(#3,#1)
 \setcounter{PartitionLength}{#3-#2}
 \setcounter{PartitionDepth}{#1}
 \put(#2,0){\line(0,1){#1}}     
 \put(#3,0){\line(0,1){#1}}
 \put(#2,\thePartitionDepth){\line(1,0){\thePartitionLength}}
 \end{picture}}
\newcommand{\uppartiii}[4]{
 \begin{picture}(#4,#1)
 \setcounter{PartitionLength}{#4-#2}
 \setcounter{PartitionDepth}{#1}
 \put(#2,0){\line(0,1){#1}}
 \put(#3,0){\line(0,1){#1}}
 \put(#4,0){\line(0,1){#1}}
 \put(#2,\thePartitionDepth){\line(1,0){\thePartitionLength}} 
 \end{picture}}

\newsavebox{\boxpaarpart}
   \savebox{\boxpaarpart}
   { \begin{picture}(0.7,0.5)
     \put(0,0){\line(0,1){0.4}}
     \put(0.5,0){\line(0,1){0.4}}
     \put(0,0.4){\line(1,0){0.5}}
     \end{picture}}
\newsavebox{\boxbaarpart}
   \savebox{\boxbaarpart}
   { \begin{picture}(0.7,0.5)
     \put(0,0){\line(0,1){0.4}}
     \put(0.5,0){\line(0,1){0.4}}
     \put(0,0){\line(1,0){0.5}}
     \end{picture}}
\newsavebox{\boxdreipart}
   \savebox{\boxdreipart}
   { \begin{picture}(1,0.5)
     \put(0,0){\line(0,1){0.4}}
     \put(0.4,0){\line(0,1){0.4}}
     \put(0.8,0){\line(0,1){0.4}}
     \put(0,0.4){\line(1,0){0.8}}
     \end{picture}}
\newsavebox{\boxvierpart}
   \savebox{\boxvierpart}
   { \begin{picture}(1.4,0.5)
     \put(0,0){\line(0,1){0.4}}
     \put(0.4,0){\line(0,1){0.4}}
     \put(0.8,0){\line(0,1){0.4}}
     \put(1.2,0){\line(0,1){0.4}}
     \put(0,0.4){\line(1,0){1.2}}
     \end{picture}}
\newsavebox{\boxvierpartrot}
   \savebox{\boxvierpartrot}
   { \begin{picture}(0.5,0.8)
     \put(0,-0.2){\line(0,1){0.3}}
     \put(0.4,-0.2){\line(0,1){0.3}}
     \put(0,0.1){\line(1,0){0.4}}
     \put(0,0.4){\line(0,1){0.3}}
     \put(0.4,0.4){\line(0,1){0.3}}
     \put(0,0.4){\line(1,0){0.4}}
     \put(0.2,0.1){\line(0,1){0.3}}
     \end{picture}}
\newsavebox{\boxvierpartrotdrei}
   \savebox{\boxvierpartrotdrei}
   { \begin{picture}(1,0.8)
     \put(0,-0.2){\line(0,1){0.3}}
     \put(0.4,-0.2){\line(0,1){0.8}}
     \put(0.8,-0.2){\line(0,1){0.3}}
     \put(0,0.1){\line(1,0){0.8}}
     \end{picture}}
\newsavebox{\boxcrosspart}
   \savebox{\boxcrosspart}
   { \begin{picture}(0.5,0.8)
     \put(0,-0.2){\line(1,2){0.4}}
     \put(0.4,-0.2){\line(-1,2){0.4}}
     \end{picture}}
\newsavebox{\boxhalflibpart}
   \savebox{\boxhalflibpart}
   { \begin{picture}(1,0.8)
     \put(0,-0.2){\line(1,1){0.8}}
     \put(0.8,-0.2){\line(-1,1){0.8}}
     \put(0.4,-0.2){\line(0,1){0.8}}
     \end{picture}}
\newsavebox{\boxpositioner} 
   \savebox{\boxpositioner}
   { \begin{picture}(1.4,0.5)
     \put(0,0){\line(0,1){0.3}}
     \put(0.4,0){\line(0,1){0.6}}
     \put(0.8,0){\line(0,1){0.3}}
     \put(1.2,0){\line(0,1){0.6}}
     \put(0.4,0.6){\line(1,0){0.8}}
     \end{picture}}
\newsavebox{\boxfatcross} 
   \savebox{\boxfatcross}
   { \begin{picture}(1.4,1.3)
     \put(0,-0.2){\line(0,1){0.3}}
     \put(0.4,-0.2){\line(0,1){0.3}}
     \put(0.8,-0.2){\line(0,1){0.3}}
     \put(1.2,-0.2){\line(0,1){0.3}}
     \put(0,0.1){\line(1,0){0.4}}
     \put(0.8,0.1){\line(1,0){0.4}}
     \put(0,0.9){\line(0,1){0.3}}
     \put(0.4,0.9){\line(0,1){0.3}}
     \put(0.8,0.9){\line(0,1){0.3}}
     \put(1.2,0.9){\line(0,1){0.3}}
     \put(0,0.9){\line(1,0){0.4}}
     \put(0.8,0.9){\line(1,0){0.4}}
     \put(0.2,0.1){\line(1,1){0.8}}
     \put(0.2,0.9){\line(1,-1){0.8}}
     \end{picture}}
\newsavebox{\boxprimarypart} 
   \savebox{\boxprimarypart}
   { \begin{picture}(1,1.3)
     \put(0,-0.2){\line(0,1){0.3}}
     \put(0.4,-0.2){\line(0,1){0.3}}
     \put(0.8,-0.2){\line(0,1){0.3}}
     \put(0.4,0.1){\line(1,0){0.4}}
     \put(0,0.9){\line(0,1){0.3}}
     \put(0.4,0.9){\line(0,1){0.3}}
     \put(0.8,0.9){\line(0,1){0.3}}
     \put(0,0.9){\line(1,0){0.4}}
     \put(0,0.1){\line(1,1){0.8}}
     \put(0.2,0.9){\line(1,-2){0.4}}
     \end{picture}}
\newcommand{\paarpart}{\usebox{\boxpaarpart}}

\newcommand{\vierpart}{\usebox{\boxvierpart}}

\newcommand{\legpart}{\usebox{\boxpositioner}}

\newcommand{\singleton}{\mathord{\uparrow}}

\newcommand{\twocol}{{\circ\bullet}}

\newsavebox{\boxidpartww}
   \savebox{\boxidpartww}
   { \begin{picture}(0.5,1.2)
     \put(0.2,0.15){\line(0,1){0.7}}
     \put(0,-0.2){$\circ$}
     \put(0,0.8){$\circ$}
     \end{picture}}
\newsavebox{\boxidpartbw}
   \savebox{\boxidpartbw}
   { \begin{picture}(0.5,1.2)
     \put(0.2,0.15){\line(0,1){0.7}}
     \put(0,-0.2){$\circ$}
     \put(0,0.8){$\bullet$}
     \end{picture}}
\newsavebox{\boxidpartwb}
   \savebox{\boxidpartwb}
   { \begin{picture}(0.5,1.2)
     \put(0.2,0.15){\line(0,1){0.7}}
     \put(0,-0.2){$\bullet$}
     \put(0,0.8){$\circ$}
     \end{picture}}
\newsavebox{\boxidpartbb}
   \savebox{\boxidpartbb}
   { \begin{picture}(0.5,1.2)
     \put(0.2,0.15){\line(0,1){0.7}}
     \put(0,-0.2){$\bullet$}
     \put(0,0.8){$\bullet$}
     \end{picture}}

\newsavebox{\boxidpartsingletonbb}
   \savebox{\boxidpartsingletonbb}
   { \begin{picture}(0.5,1.2)
     \put(0.2,0.15){\line(0,1){0.2}}
     \put(0.2,0.65){\line(0,1){0.2}}
     \put(0,-0.2){$\bullet$}
     \put(0,0.8){$\bullet$}
     \end{picture}}
\newsavebox{\boxidpartsingletonww}
   \savebox{\boxidpartsingletonww}
   { \begin{picture}(0.5,1.2)
     \put(0.2,0.15){\line(0,1){0.2}}
     \put(0.2,0.65){\line(0,1){0.2}}
     \put(0,-0.2){$\circ$}
     \put(0,0.8){$\circ$}
     \end{picture}}
     
\newsavebox{\boxpaarpartbb}
   \savebox{\boxpaarpartbb}
   { \begin{picture}(0.8,0.7)
     \put(0.0,0.2){\line(0,1){0.4}}
     \put(0.5,0.2){\line(0,1){0.4}}
     \put(0.0,0.6){\line(1,0){0.5}}
     \put(0.3,-0.2){$\bullet$}
     \put(-0.2,-0.2){$\bullet$}
     \end{picture}}
\newsavebox{\boxpaarpartww}
   \savebox{\boxpaarpartww}
   { \begin{picture}(0.8,0.7)
     \put(0.0,0.2){\line(0,1){0.4}}
     \put(0.5,0.2){\line(0,1){0.4}}
     \put(0.0,0.6){\line(1,0){0.5}}
     \put(0.3,-0.2){$\circ$}
     \put(-0.2,-0.2){$\circ$}
     \end{picture}}
\newsavebox{\boxpaarpartbw}
   \savebox{\boxpaarpartbw}
   { \begin{picture}(0.8,0.7)
     \put(0.0,0.2){\line(0,1){0.4}}
     \put(0.5,0.2){\line(0,1){0.4}}
     \put(0.0,0.6){\line(1,0){0.5}}
     \put(0.3,-0.2){$\circ$}
     \put(-0.2,-0.2){$\bullet$}
     \end{picture}}
\newsavebox{\boxpaarpartwb}
   \savebox{\boxpaarpartwb}
   { \begin{picture}(0.8,0.7)
     \put(0.0,0.2){\line(0,1){0.4}}
     \put(0.5,0.2){\line(0,1){0.4}}
     \put(0.0,0.6){\line(1,0){0.5}}
     \put(0.3,-0.2){$\bullet$}
     \put(-0.2,-0.2){$\circ$}
     \end{picture}}
\newsavebox{\boxbaarpartbb}
   \savebox{\boxbaarpartbb}
   { \begin{picture}(1,0.7)
     \put(0.2,-0.1){\line(0,1){0.4}}
     \put(0.7,-0.1){\line(0,1){0.4}}
     \put(0.2,-0.1){\line(1,0){0.5}}
     \put(0.5,0.3){$\bullet$}
     \put(0,0.3){$\bullet$}
     \end{picture}}
\newsavebox{\boxbaarpartww}
   \savebox{\boxbaarpartww}
   { \begin{picture}(1,0.7)
     \put(0.2,-0.1){\line(0,1){0.4}}
     \put(0.7,-0.1){\line(0,1){0.4}}
     \put(0.2,-0.1){\line(1,0){0.5}}
     \put(0.5,0.3){$\circ$}
     \put(0,0.3){$\circ$}
     \end{picture}}
\newsavebox{\boxbaarpartbw}
   \savebox{\boxbaarpartbw}
   { \begin{picture}(1,0.7)
     \put(0.2,-0.1){\line(0,1){0.4}}
     \put(0.7,-0.1){\line(0,1){0.4}}
     \put(0.2,-0.1){\line(1,0){0.5}}
     \put(0.5,0.3){$\circ$}
     \put(0,0.3){$\bullet$}
     \end{picture}}
\newsavebox{\boxbaarpartwb}
   \savebox{\boxbaarpartwb}
   { \begin{picture}(1,0.7)
     \put(0.2,-0.1){\line(0,1){0.4}}
     \put(0.7,-0.1){\line(0,1){0.4}}
     \put(0.2,-0.1){\line(1,0){0.5}}
     \put(0.5,0.3){$\bullet$}
     \put(0,0.3){$\circ$}
     \end{picture}}

\newsavebox{\boxcutpaarpartbb}
   \savebox{\boxcutpaarpartbb}
   { \begin{picture}(1,0.7)
     \put(0,0.2){\line(0,1){0.4}}
     \put(0.8,0.2){\line(0,1){0.4}}
     \put(0,0.6){\line(1,0){0.2}}
     \put(0.6,0.6){\line(1,0){0.2}}
     \put(0.6,-0.2){$\bullet$}
     \put(-0.2,-0.2){$\bullet$}
     \end{picture}}
\newsavebox{\boxcutpaarpartww}
   \savebox{\boxcutpaarpartww}
   { \begin{picture}(1,0.7)
     \put(0,0.2){\line(0,1){0.4}}
     \put(0.8,0.2){\line(0,1){0.4}}
     \put(0,0.6){\line(1,0){0.2}}
     \put(0.6,0.6){\line(1,0){0.2}}
     \put(0.6,-0.2){$\circ$}
     \put(-0.2,-0.2){$\circ$}
     \end{picture}}
\newsavebox{\boxcutpaarpartbw}
   \savebox{\boxcutpaarpartbw}
   { \begin{picture}(1,0.7)
     \put(0,0.2){\line(0,1){0.4}}
     \put(0.8,0.2){\line(0,1){0.4}}
     \put(0,0.6){\line(1,0){0.2}}
     \put(0.6,0.6){\line(1,0){0.2}}
     \put(0.6,-0.2){$\circ$}
     \put(-0.2,-0.2){$\bullet$}
     \end{picture}}
\newsavebox{\boxcutpaarpartwb}
   \savebox{\boxcutpaarpartwb}
   { \begin{picture}(1,0.7)
     \put(0,0.2){\line(0,1){0.4}}
     \put(0.8,0.2){\line(0,1){0.4}}
     \put(0,0.6){\line(1,0){0.2}}
     \put(0.6,0.6){\line(1,0){0.2}}
     \put(0.6,-0.2){$\bullet$}
     \put(-0.2,-0.2){$\circ$}
     \end{picture}}

\newsavebox{\boxsingletonw}
   \savebox{\boxsingletonw}
   { \symstrut 0.2cm
     \begin{picture}(0.3, 0.8)
     \put(-0.2,0.2){$\uparrow$}
     \put(-0.2,-0.3){$\circ$}
     \end{picture}}
\newsavebox{\boxsingletonb}
   \savebox{\boxsingletonb}
   { \symstrut 0.2cm
     \begin{picture}(0.3, 0.8)
     \put(-0.2,0.2){$\uparrow$}
     \put(-0.2,-0.3){$\bullet$}
     \end{picture}}
\newsavebox{\boxdownsingletonw}
   \savebox{\boxdownsingletonw}
   { \begin{picture}(0.5, 1)
     \put(0,0){$\downarrow$}
     \put(0,0.6){$\circ$}
     \end{picture}}
\newsavebox{\boxdownsingletonb}
   \savebox{\boxdownsingletonb}
   { \begin{picture}(0.5, 1)
     \put(0,0){$\downarrow$}
     \put(0,0.6){$\bullet$}
     \end{picture}}

\newsavebox{\boxvierpartwbwb}
   \savebox{\boxvierpartwbwb}
   { \begin{picture}(1.5,0.7)
     \put(0.0,0.2){\line(0,1){0.4}}
     \put(0.4,0.2){\line(0,1){0.4}}
     \put(0.8,0.2){\line(0,1){0.4}}
     \put(1.2,0.2){\line(0,1){0.4}}
     \put(0.0,0.6){\line(1,0){1.2}}
     \put(-0.2,-0.2){$\circ$}
     \put(0.2,-0.2){$\bullet$}
     \put(0.6,-0.2){$\circ$}
     \put(1.0,-0.2){$\bullet$}
     \end{picture}}
\newsavebox{\boxvierpartbwbw}
   \savebox{\boxvierpartbwbw}
   { \begin{picture}(1.7,0.7)
     \put(0.0,0.2){\line(0,1){0.4}}
     \put(0.4,0.2){\line(0,1){0.4}}
     \put(0.8,0.2){\line(0,1){0.4}}
     \put(1.2,0.2){\line(0,1){0.4}}
     \put(0.0,0.6){\line(1,0){1.2}}
     \put(-0.2,-0.2){$\bullet$}
     \put(0.2,-0.2){$\circ$}
     \put(0.6,-0.2){$\bullet$}
     \put(1.0,-0.2){$\circ$}
     \end{picture}}
\newsavebox{\boxvierpartwwbb}
   \savebox{\boxvierpartwwbb}
   { \begin{picture}(1.7,0.7)
     \put(0.0,0.2){\line(0,1){0.4}}
     \put(0.4,0.2){\line(0,1){0.4}}
     \put(0.8,0.2){\line(0,1){0.4}}
     \put(1.2,0.2){\line(0,1){0.4}}
     \put(0.0,0.6){\line(1,0){1.2}}
     \put(-0.2,-0.2){$\circ$}
     \put(0.2,-0.2){$\circ$}
     \put(0.6,-0.2){$\bullet$}
     \put(1.0,-0.2){$\bullet$}
     \end{picture}}
\newsavebox{\boxvierpartwbbw}
   \savebox{\boxvierpartwbbw}
   { \begin{picture}(1.7,0.7)
     \put(0.0,0.2){\line(0,1){0.4}}
     \put(0.4,0.2){\line(0,1){0.4}}
     \put(0.8,0.2){\line(0,1){0.4}}
     \put(1.2,0.2){\line(0,1){0.4}}
     \put(0.0,0.6){\line(1,0){1.2}}
     \put(-0.2,-0.2){$\circ$}
     \put(0.2,-0.2){$\bullet$}
     \put(0.6,-0.2){$\bullet$}
     \put(1.0,-0.2){$\circ$}
     \end{picture}}
\newsavebox{\boxvierpartbwwb}
   \savebox{\boxvierpartbwwb}
   { \begin{picture}(1.7,0.7)
     \put(0.0,0.2){\line(0,1){0.4}}
     \put(0.4,0.2){\line(0,1){0.4}}
     \put(0.8,0.2){\line(0,1){0.4}}
     \put(1.2,0.2){\line(0,1){0.4}}
     \put(0.0,0.6){\line(1,0){1.2}}
     \put(-0.2,-0.2){$\bullet$}
     \put(0.2,-0.2){$\circ$}
     \put(0.6,-0.2){$\circ$}
     \put(1.0,-0.2){$\bullet$}
     \end{picture}}
\newsavebox{\boxvierpartrotwbwb}
   \savebox{\boxvierpartrotwbwb}
   { \begin{picture}(1,1.2)
     \put(0.2,0.15){\line(0,1){0.2}}
     \put(0.2,0.65){\line(0,1){0.2}}
     \put(0.2,0.65){\line(1,0){0.5}}
     \put(0.2,0.35){\line(1,0){0.5}}
     \put(0.45,0.35){\line(0,1){0.3}}
     \put(0.7,0.15){\line(0,1){0.2}}
     \put(0.7,0.65){\line(0,1){0.2}}
     \put(0,0.8){$\circ$}
     \put(0.5,0.8){$\bullet$}
     \put(0,-0.2){$\circ$}
     \put(0.5,-0.2){$\bullet$}
     \end{picture}}
\newsavebox{\boxvierpartrotbwwb}
   \savebox{\boxvierpartrotbwwb}
   { \begin{picture}(1,1.2)
     \put(0.2,0.15){\line(0,1){0.2}}
     \put(0.2,0.65){\line(0,1){0.2}}
     \put(0.2,0.65){\line(1,0){0.5}}
     \put(0.2,0.35){\line(1,0){0.5}}
     \put(0.45,0.35){\line(0,1){0.3}}
     \put(0.7,0.15){\line(0,1){0.2}}
     \put(0.7,0.65){\line(0,1){0.2}}
     \put(0,0.8){$\bullet$}
     \put(0.5,0.8){$\circ$}
     \put(0,-0.2){$\circ$}
     \put(0.5,-0.2){$\bullet$}
     \end{picture}}
\newsavebox{\boxvierpartrotbwbw}
   \savebox{\boxvierpartrotbwbw}
   { \begin{picture}(1,1.2)
     \put(0.2,0.15){\line(0,1){0.2}}
     \put(0.2,0.65){\line(0,1){0.2}}
     \put(0.2,0.65){\line(1,0){0.5}}
     \put(0.2,0.35){\line(1,0){0.5}}
     \put(0.45,0.35){\line(0,1){0.3}}
     \put(0.7,0.15){\line(0,1){0.2}}
     \put(0.7,0.65){\line(0,1){0.2}}
     \put(0,0.8){$\bullet$}
     \put(0.5,0.8){$\circ$}
     \put(0,-0.2){$\bullet$}
     \put(0.5,-0.2){$\circ$}
     \end{picture}}
\newsavebox{\boxvierpartrotwwww}
   \savebox{\boxvierpartrotwwww}
   { \begin{picture}(1,1.2)
     \put(0.2,0.15){\line(0,1){0.2}}
     \put(0.2,0.65){\line(0,1){0.2}}
     \put(0.2,0.65){\line(1,0){0.5}}
     \put(0.2,0.35){\line(1,0){0.5}}
     \put(0.45,0.35){\line(0,1){0.3}}
     \put(0.7,0.15){\line(0,1){0.2}}
     \put(0.7,0.65){\line(0,1){0.2}}
     \put(0,0.8){$\circ$}
     \put(0.5,0.8){$\circ$}
     \put(0,-0.2){$\circ$}
     \put(0.5,-0.2){$\circ$}
     \end{picture}}
\newsavebox{\boxvierpartrotbbbb}
   \savebox{\boxvierpartrotbbbb}
   { \begin{picture}(1,1.2)
     \put(0.2,0.15){\line(0,1){0.2}}
     \put(0.2,0.65){\line(0,1){0.2}}
     \put(0.2,0.65){\line(1,0){0.5}}
     \put(0.2,0.35){\line(1,0){0.5}}
     \put(0.45,0.35){\line(0,1){0.3}}
     \put(0.7,0.15){\line(0,1){0.2}}
     \put(0.7,0.65){\line(0,1){0.2}}
     \put(0,0.8){$\bullet$}
     \put(0.5,0.8){$\bullet$}
     \put(0,-0.2){$\bullet$}
     \put(0.5,-0.2){$\bullet$}
     \end{picture}}

\newsavebox{\boxdreipartwww}
   \savebox{\boxdreipartwww}
   { \begin{picture}(1,0.7)
     \put(0.0,0.2){\line(0,1){0.4}}
     \put(0.4,0.2){\line(0,1){0.4}}
     \put(0.8,0.2){\line(0,1){0.4}}
     \put(0.0,0.6){\line(1,0){0.8}}
     \put(-0.2,-0.2){$\circ$}
     \put(0.2,-0.2){$\circ$}
     \put(0.6,-0.2){$\circ$}
     \end{picture}}
\newsavebox{\boxdreipartbbb}
   \savebox{\boxdreipartbbb}
   { \begin{picture}(1,0.7)
     \put(0.0,0.2){\line(0,1){0.4}}
     \put(0.4,0.2){\line(0,1){0.4}}
     \put(0.8,0.2){\line(0,1){0.4}}
     \put(0.0,0.6){\line(1,0){0.8}}
     \put(-0.2,-0.2){$\bullet$}
     \put(0.2,-0.2){$\bullet$}
     \put(0.6,-0.2){$\bullet$}
     \end{picture}}

\newsavebox{\boxsechspartwbwbwb}
   \savebox{\boxsechspartwbwbwb}
   { \begin{picture}(2.5,0.7)
     \put(0.2,0.2){\line(0,1){0.4}}
     \put(0.6,0.2){\line(0,1){0.4}}
     \put(1,0.2){\line(0,1){0.4}}
     \put(1.4,0.2){\line(0,1){0.4}}
     \put(1.8,0.2){\line(0,1){0.4}}
     \put(2.2,0.2){\line(0,1){0.4}}
     \put(0.2,0.6){\line(1,0){2}}
     \put(0,-0.2){$\circ$}
     \put(0.4,-0.2){$\bullet$}
     \put(0.8,-0.2){$\circ$}
     \put(1.2,-0.2){$\bullet$}
     \put(1.6,-0.2){$\circ$}
     \put(2.0,-0.2){$\bullet$}
     \end{picture}}
 
\newsavebox{\boxcrosspartwbbw}
   \savebox{\boxcrosspartwbbw}
   { \symstrut 0.25cm
     \begin{picture}(1,1.1)(0,0.3)
     \put(0.25,0.15){\line(1,2){0.4}}
     \put(0.65,0.15){\line(-1,2){0.4}}
     \put(0,0.9){$\circ$}
     \put(0.5,0.9){$\bullet$}
     \put(0,-0.2){$\bullet$}
     \put(0.5,-0.2){$\circ$}
     \end{picture}}
\newsavebox{\boxcrosspartbwwb}
   \savebox{\boxcrosspartbwwb}
   { \symstrut 0.25cm
     \begin{picture}(1,1.1)(0,0.3)
     \put(0.25,0.15){\line(1,2){0.4}}
     \put(0.65,0.15){\line(-1,2){0.4}}
     \put(0,0.9){$\bullet$}
     \put(0.5,0.9){$\circ$}
     \put(0,-0.2){$\circ$}
     \put(0.5,-0.2){$\bullet$}
     \end{picture}}
\newsavebox{\boxcrosspartwwww}
   \savebox{\boxcrosspartwwww}
   { \symstrut 0.25cm
     \begin{picture}(1,1.1)(0,0.3)
     \put(0.25,0.15){\line(1,2){0.4}}
     \put(0.65,0.15){\line(-1,2){0.4}}
     \put(0,0.9){$\circ$}
     \put(0.5,0.9){$\circ$}
     \put(0,-0.2){$\circ$}
     \put(0.5,-0.2){$\circ$}
     \end{picture}}
\newsavebox{\boxcrosspartbbbb}
   \savebox{\boxcrosspartbbbb}
   { \symstrut 0.25cm
     \begin{picture}(1,1.1)(0,0.3)
     \put(0.25,0.15){\line(1,2){0.4}}
     \put(0.65,0.15){\line(-1,2){0.4}}
     \put(0,0.9){$\bullet$}
     \put(0.5,0.9){$\bullet$}
     \put(0,-0.2){$\bullet$}
     \put(0.5,-0.2){$\bullet$}
     \end{picture}}

\newsavebox{\boxhalflibpartwwwwww}
   \savebox{\boxhalflibpartwwwwww}
   { \symstrut 0.25cm
     \begin{picture}(1.5,1.1)(0,0.3)
     \put(0.3,0.15){\line(1,1){0.8}}
     \put(1.1,0.15){\line(-1,1){0.8}}
     \put(0.7,0.15){\line(0,1){0.8}}     
     \put(0,0.9){$\circ$}
     \put(0.5,0.9){$\circ$}
     \put(1,0.9){$\circ$}
     \put(0,-0.2){$\circ$}
     \put(0.5,-0.2){$\circ$}
     \put(1,-0.2){$\circ$}     
     \end{picture}}

\newsavebox{\boxpositionerd} 
   \savebox{\boxpositionerd}
   { \begin{picture}(2.6,1.5)
     \put(0.9,0.2){\line(0,1){0.9}}
     \put(2.3,0.2){\line(0,1){0.9}}
     \put(0.9,1.1){\line(1,0){1.4}}
     \put(0.05,-0.05){$^\uparrow$}
     \put(0.2,0.4){\tiny$^{\otimes d}$}
     \put(1.35,-0.05){$^\uparrow$}
     \put(1.5,0.4){\tiny$^{\otimes d}$}
     \put(0,-0.2){$\circ$}
     \put(0.7,-0.2){$\circ$}
     \put(1.3,-0.2){$\bullet$}
     \put(2.1,-0.2){$\bullet$}
     \end{picture}}
\newsavebox{\boxpositionerrpluseins} 
   \savebox{\boxpositionerrpluseins}
   { \begin{picture}(3.9,1.5)
     \put(1.6,0.2){\line(0,1){0.9}}
     \put(3.6,0.2){\line(0,1){0.9}}
     \put(1.6,1.1){\line(1,0){2}}
     \put(0.05,-0.05){$^\uparrow$}
     \put(0.2,0.4){\tiny$^{\otimes r+1}$}
     \put(2.05,-0.05){$^\uparrow$}
     \put(2.2,0.4){\tiny$^{\otimes r-1}$}
     \put(0,-0.2){$\circ$}
     \put(1.4,-0.2){$\bullet$}
     \put(2,-0.2){$\bullet$}
     \put(3.4,-0.2){$\bullet$}
     \end{picture}}
\newsavebox{\boxpositionerdt} 
   \savebox{\boxpositionerdt}
   { \begin{picture}(2.9,1.5)
     \put(1.2,0.2){\line(0,1){0.9}}
     \put(2.9,0.2){\line(0,1){0.9}}
     \put(1.2,1.1){\line(1,0){1.7}}
     \put(0.05,-0.05){$^\uparrow$}
     \put(0.2,0.4){\tiny$^{\otimes dt}$}
     \put(1.65,-0.05){$^\uparrow$}
     \put(1.8,0.4){\tiny$^{\otimes dt}$}
     \put(0,-0.2){$\circ$}
     \put(1,-0.2){$\circ$}
     \put(1.6,-0.2){$\bullet$}
     \put(2.7,-0.2){$\bullet$}
     \end{picture}}
\newsavebox{\boxpositioners} 
   \savebox{\boxpositioners}
   { \begin{picture}(2.6,1.5)
     \put(0.9,0.2){\line(0,1){0.9}}
     \put(2.3,0.2){\line(0,1){0.9}}
     \put(0.9,1.1){\line(1,0){1.4}}
     \put(0.05,-0.05){$^\uparrow$}
     \put(0.2,0.4){\tiny$^{\otimes s}$}
     \put(1.35,-0.05){$^\uparrow$}
     \put(1.5,0.4){\tiny$^{\otimes s}$}
     \put(0,-0.2){$\circ$}
     \put(0.7,-0.2){$\circ$}
     \put(1.3,-0.2){$\bullet$}
     \put(2.1,-0.2){$\bullet$}
     \end{picture}}
\newsavebox{\boxpositionerdinv} 
   \savebox{\boxpositionerdinv}
   { \begin{picture}(2.6,1.5)
     \put(0.9,0.2){\line(0,1){0.9}}
     \put(2.3,0.2){\line(0,1){0.9}}
     \put(0.9,1.1){\line(1,0){1.4}}
     \put(0.05,-0.05){$^\uparrow$}
     \put(0.2,0.4){\tiny$^{\otimes d}$}
     \put(1.35,-0.05){$^\uparrow$}
     \put(1.5,0.4){\tiny$^{\otimes d}$}
     \put(0,-0.2){$\circ$}
     \put(0.7,-0.2){$\bullet$}
     \put(1.3,-0.2){$\bullet$}
     \put(2.1,-0.2){$\circ$}
     \end{picture}}
\newsavebox{\boxpositionersinv} 
   \savebox{\boxpositionersinv}
   { \begin{picture}(2.6,1.5)
     \put(0.9,0.2){\line(0,1){0.9}}
     \put(2.3,0.2){\line(0,1){0.9}}
     \put(0.9,1.1){\line(1,0){1.4}}
     \put(0.05,-0.05){$^\uparrow$}
     \put(0.2,0.4){\tiny$^{\otimes s}$}
     \put(1.35,-0.05){$^\uparrow$}
     \put(1.5,0.4){\tiny$^{\otimes s}$}
     \put(0,-0.2){$\circ$}
     \put(0.7,-0.2){$\bullet$}
     \put(1.3,-0.2){$\bullet$}
     \put(2.1,-0.2){$\circ$}
     \end{picture}}
\newsavebox{\boxpositionerdpluszwei}
   \savebox{\boxpositionerdpluszwei}
   { \begin{picture}(3.4,1.5)
     \put(1.6,0.2){\line(0,1){0.9}}
     \put(3.0,0.2){\line(0,1){0.9}}
     \put(1.6,1.1){\line(1,0){1.4}}
     \put(0.05,-0.05){$^\uparrow$}
     \put(0.2,0.4){\tiny$^{\otimes d+2}$}
     \put(2.05,-0.05){$^\uparrow$}
     \put(2.2,0.4){\tiny$^{\otimes d}$}
     \put(0,-0.2){$\circ$}
     \put(1.4,-0.2){$\bullet$}
     \put(2,-0.2){$\bullet$}
     \put(2.8,-0.2){$\bullet$}
     \end{picture}}
\newsavebox{\boxpositionersminuszwei}
   \savebox{\boxpositionersminuszwei}
   { \begin{picture}(3.4,1.5)
     \put(1,0.2){\line(0,1){0.9}}
     \put(3.0,0.2){\line(0,1){0.9}}
     \put(1,1.1){\line(1,0){2}}
     \put(0.05,-0.05){$^\uparrow$}
     \put(0.2,0.4){\tiny$^{\otimes s}$}
     \put(1.45,-0.05){$^\uparrow$}
     \put(1.6,0.4){\tiny$^{\otimes s-2}$}
     \put(0,-0.2){$\circ$}
     \put(0.8,-0.2){$\bullet$}
     \put(1.4,-0.2){$\bullet$}
     \put(2.8,-0.2){$\bullet$}
     \end{picture}}
\newsavebox{\boxpositionerrnull}
   \savebox{\boxpositionerrnull}
   { \begin{picture}(3.8,1.5)
     \put(1.2,0.2){\line(0,1){0.9}}
     \put(3.4,0.2){\line(0,1){0.9}}
     \put(1.2,1.1){\line(1,0){2.2}}
     \put(0.05,-0.05){$^\uparrow$}
     \put(0.2,0.4){\tiny$^{\otimes r_0}$}
     \put(1.65,-0.05){$^\uparrow$}
     \put(1.8,0.4){\tiny$^{\otimes r_0-2}$}
     \put(0,-0.2){$\circ$}
     \put(1,-0.2){$\bullet$}
     \put(1.6,-0.2){$\bullet$}
     \put(3.2,-0.2){$\bullet$}
     \end{picture}}
\newsavebox{\boxpositionerwbwb}
   \savebox{\boxpositionerwbwb}
   { \begin{picture}(1.8,1)
     \put(0.6,0.2){\line(0,1){0.6}}
     \put(1.4,0.2){\line(0,1){0.6}}
     \put(0.6,0.8){\line(1,0){0.8}}
     \put(0.05,-0.05){$^\uparrow$}
     \put(0.85,-0.05){$^\uparrow$}
     \put(0,-0.2){$\circ$}
     \put(0.4,-0.2){$\bullet$}
     \put(0.8,-0.2){$\circ$}
     \put(1.2,-0.2){$\bullet$}
     \end{picture}}
\newsavebox{\boxpositionerwwbb}
   \savebox{\boxpositionerwwbb}
   { \begin{picture}(1.8,1)
     \put(0.6,0.2){\line(0,1){0.6}}
     \put(1.4,0.2){\line(0,1){0.6}}
     \put(0.6,0.8){\line(1,0){0.8}}
     \put(0.05,-0.05){$^\uparrow$}
     \put(0.85,-0.05){$^\uparrow$}
     \put(0,-0.2){$\circ$}
     \put(0.4,-0.2){$\circ$}
     \put(0.8,-0.2){$\bullet$}
     \put(1.2,-0.2){$\bullet$}
     \end{picture}}
\newsavebox{\boxpositionerrevwbwb}
   \savebox{\boxpositionerrevwbwb}
   { \begin{picture}(1.8,1)
     \put(0.2,0.2){\line(0,1){0.6}}
     \put(1.0,0.2){\line(0,1){0.6}}
     \put(0.2,0.8){\line(1,0){0.8}}
     \put(0.45,-0.05){$^\uparrow$}
     \put(1.25,-0.05){$^\uparrow$}
     \put(0,-0.2){$\circ$}
     \put(0.4,-0.2){$\bullet$}
     \put(0.8,-0.2){$\circ$}
     \put(1.2,-0.2){$\bullet$}
     \end{picture}}
\newsavebox{\boxpositionersalphaw} 
   \savebox{\boxpositionersalphaw}
   { \begin{picture}(2.6,1.5)
     \put(0.9,0.2){\line(0,1){0.9}}
     \put(2.3,0.2){\line(0,1){0.9}}
     \put(0.9,1.1){\line(1,0){1.4}}
     \put(0.05,-0.05){$^\uparrow$}
     \put(1.35,-0.05){$^\uparrow$}
     \put(0.2,0.4){\tiny$^{\otimes s}$}
     \put(1.5,0.4){\tiny$^{\otimes \alpha}$}
     \put(0,-0.2){$\circ$}
     \put(0.7,-0.2){$\circ$}
     \put(1.3,-0.2){$\circ$}
     \put(2.1,-0.2){$\bullet$}
     \end{picture}}
\newsavebox{\boxpositionersalphab} 
   \savebox{\boxpositionersalphab}
   { \begin{picture}(2.6,1.5)
     \put(0.9,0.2){\line(0,1){0.9}}
     \put(2.3,0.2){\line(0,1){0.9}}
     \put(0.9,1.1){\line(1,0){1.4}}
     \put(0.05,-0.05){$^\uparrow$}
     \put(1.35,-0.05){$^\uparrow$}
     \put(0.2,0.4){\tiny$^{\otimes s}$}
     \put(1.5,0.4){\tiny$^{\otimes \alpha}$}
     \put(0,-0.2){$\circ$}
     \put(0.7,-0.2){$\circ$}
     \put(1.3,-0.2){$\bullet$}
     \put(2.1,-0.2){$\bullet$}
     \end{picture}}

\newsavebox{\boxAspace}
   \savebox{\boxAspace}
   { \begin{picture}(0.1,1.5)
     \end{picture}}
\newsavebox{\boxBspace}
   \savebox{\boxBspace}
   { \begin{picture}(0,1.85)
     \end{picture}}

\newcommand{\paarpartww}{\usebox{\boxpaarpartww}}
\newcommand{\paarpartbw}{\usebox{\boxpaarpartbw}}
\newcommand{\paarpartwb}{\usebox{\boxpaarpartwb}}
\newcommand{\paarpartbb}{\usebox{\boxpaarpartbb}}

\newcommand{\singletonw}{\usebox{\boxsingletonw}}
\newcommand{\singletonb}{\usebox{\boxsingletonb}}

\newcommand{\vierpartwbwb}{\usebox{\boxvierpartwbwb}}

\newcommand{\dreipartbbb}{\usebox{\boxdreipartbbb}}

\newcommand{\positionerwwbb}{\usebox{\boxpositionerwwbb}}

\DeclareMathOperator{\id}{id}

\newcommand{\tensorglued}{\mathbin{\tilde\times}}

\newsavebox{\boxcrosspartrot}
   \savebox{\boxcrosspartrot}
   { \begin{picture}(1.5,0.5)
     \put(0,-0.1){\line(0,1){0.4}}
     \put(0.8,-0.1){\line(0,1){0.4}}
     \put(0,0.3){\line(1,0){0.8}}
     \put(0.4,-0.1){\line(0,1){0.6}}
     \put(1.2,-0.1){\line(0,1){0.6}}
     \put(0.4,0.5){\line(1,0){0.8}}
     \end{picture}}

\newsavebox{\boxhalflibpartrot}
   \savebox{\boxhalflibpartrot}
   { \begin{picture}(2.3,0.5)
     \put(0,-0.2){\line(0,1){0.4}}
     \put(1.2,-0.2){\line(0,1){0.4}}
     \put(0,0.2){\line(1,0){1.2}}
     \put(0.4,-0.2){\line(0,1){0.6}}
     \put(1.6,-0.2){\line(0,1){0.6}}
     \put(0.4,0.4){\line(1,0){1.2}}
     \put(0.8,-0.2){\line(0,1){0.8}}
     \put(2.0,-0.2){\line(0,1){0.8}}
     \put(0.8,0.6){\line(1,0){1.2}}
     \end{picture}}

\newsavebox{\boxcrosspartrotwwbb}
   \savebox{\boxcrosspartrotwwbb}
   { \begin{picture}(1.5,0.5)
     \put(0,0.1){\line(0,1){0.4}}
     \put(0.8,0.1){\line(0,1){0.4}}
     \put(0,0.5){\line(1,0){0.8}}
     \put(0.4,0.1){\line(0,1){0.6}}
     \put(1.2,0.1){\line(0,1){0.6}}
     \put(0.4,0.7){\line(1,0){0.8}}
     \put(-0.2,-0.3){$\circ$}
     \put(0.2,-0.3){$\circ$}
     \put(0.6,-0.3){$\bullet$}
     \put(1.0,-0.3){$\bullet$}
     \end{picture}}

\newsavebox{\boxhalflibpartrotwwwbbb}
   \savebox{\boxhalflibpartrotwwwbbb}
   { \symstrut 0.25cm
     \begin{picture}(2.3,0.5)
     \put(0,0.0){\line(0,1){0.4}}
     \put(1.2,0.0){\line(0,1){0.4}}
     \put(0,0.4){\line(1,0){1.2}}
     \put(0.4,0.0){\line(0,1){0.6}}
     \put(1.6,0.0){\line(0,1){0.6}}
     \put(0.4,0.6){\line(1,0){1.2}}
     \put(0.8,0.0){\line(0,1){0.8}}
     \put(2.0,0.0){\line(0,1){0.8}}
     \put(0.8,0.8){\line(1,0){1.2}}
     \put(-0.2,-0.4){$\circ$}
     \put(0.2,-0.4){$\circ$}
     \put(0.6,-0.4){$\circ$}
     \put(1.0,-0.4){$\bullet$}
     \put(1.4,-0.4){$\bullet$}
     \put(1.8,-0.4){$\bullet$}
     \end{picture}}

\newcommand{\crosspartrot}{\usebox{\boxcrosspartrot}}
\newcommand{\halflibpartrot}{\usebox{\boxhalflibpartrot}}
\newcommand{\crosspartrotwwbb}{\usebox{\boxcrosspartrotwwbb}}
\newcommand{\halflibpartrotwwwbbb}{\usebox{\boxhalflibpartrotwwwbbb}}

\begin{document}
\title{Classification of globally colorized categories of partitions}
\author{Daniel Gromada}
\address{Saarland University, Fachbereich Mathematik, Postfach 151150,
66041 Saarbr\"ucken, Germany}
\email{gromada@math.uni-sb.de}
\date{\today}

\begin{abstract}
Set partitions closed under certain operations form a tensor category. They give rise to certain subgroups of the free orthogonal quantum group $O_n^+$, the so called easy quantum groups, introduced by Banica and Speicher in 2009. This correspondence was generalized to two-colored set partitions, which, in addition, assign a black or white color to each point of a set. Globally colorized categories of partitions are those categories that are invariant with respect to arbitrary permutations of colors. This article presents a classification of globally colorized categories. In addition, we show that the corresponding unitary quantum groups can be constructed from the orthogonal ones using tensor complexification.
\end{abstract}

\maketitle
\section*{Introduction}

The subject of this article is to classify certain categories of two-colored partitions as defined in \cite{tarragowebercombina}. By a \emph{partition} we mean a set partition, i.e. a decomposition of a finite set of points into disjoint subsets called \emph{blocks}. A \emph{two-colored} partition, in addition, assigns to every point of the set a black or white color. Such two-colored partitions can be represented by pictures. Below, we give two examples of two-colored partitions of seven points. Each example consists of seven points of either black or white color. The lines connect those points that are in the same block.

$$
\symstrut 0.7cm 
p=
\begin{picture}(7,1)(0,0.5)
\put(-1,0){\uppartiii{2}{1}{2}{7}}
\put(-1,0){\uppartii{1}{3}{4}}
\put(-1,0){\uppartii{1}{5}{6}}
\put(0.05,-0.35){$\bullet$}
\put(1.05,-0.35){$\circ$}
\put(2.05,-0.35){$\bullet$}
\put(3.05,-0.35){$\bullet$}
\put(4.05,-0.35){$\bullet$}
\put(5.05,-0.35){$\circ$}
\put(6.05,-0.35){$\bullet$}
\end{picture}
\qquad
q=
\begin{picture}(7,1)(0,0.5)
\put(-1,0){\uppartii{1}{1}{2}}
\put(-1,0){\uppartii{1}{3}{5}}
\put(-1,0){\uppartii{2}{4}{6}}
\put(-1,0){\upparti{1}{7}}
\put(0.05,-0.35){$\bullet$}
\put(1.05,-0.35){$\circ$}
\put(2.05,-0.35){$\circ$}
\put(3.05,-0.35){$\bullet$}
\put(4.05,-0.35){$\bullet$}
\put(5.05,-0.35){$\circ$}
\put(6.05,-0.35){$\circ$}
\end{picture}
$$

We can define certain operations on such partitions (see Section \ref{sec.operations}). If a collection of partitions is closed under those operations, we call it a \emph{category of (two-colored) partitions}.

This paper can be divided into two parts. The first one being purely combinatorial is a contribution to the classification of categories of two-colored partitions, by which it extends the work \cite{tarragowebercombina}. The second part applies the results on the theory of compact matrix quantum groups using techniques developed in \cite{tarragoweberopalg}.

Let us now briefly describe the theory of compact quantum groups, which is the main motivation for this theory. For more details, see e.g. \cite{tarragoweberopalg}. The idea behind compact (matrix) quantum groups defined by Woronowicz in \cite{woronowicz1987compact} is the following. For a compact group $G$, we can construct a commutative C*-algebra $A:=C(G)$ of continuous complex-valued functions over $G$. The multiplication $\mu\colon G\times G\to G$ can be described by a \emph{comultiplication} $\Delta\colon A\to A\otimes A$. The group axioms can also be dualized and formulated in terms of the commutative algebra $A$ and the comultiplication $\Delta$. Such an alternative definition of a group can be generalized by dropping the commutativity condition on $A$. The resulting structure called \emph{compact quantum group} forms a counterpart of groups in non-commutative geometry. In particular, generalizing compact matrix groups, we get \emph{compact matrix quantum groups}.

By some Tannaka--Krein result of Woronowicz \cite{woronowicz1988tannaka}, compact matrix quantum groups are determined by their representation theory. More precisely, the intertwiner spaces of a quantum group form a certain category. On the other hand, for any such category of intertwiners, one can construct a corresponding quantum group. In 2009, Banica and Speicher observed that such categories arise from certain categories of partitions (without colors). Thus, for any category of partitions, one can construct so called \emph{easy quantum group}. For all such easy quantum groups $G$ we have $S_n\subset G\subset O_n^+$, i.e. they are between the group of permutations $S_n$ and the free orthogonal quantum group $O_n^+$ defined by Wang in \cite{wang1995free}. 

The definition of a partition was generalized using colorings in \cite{tarragowebercombina} by Tarrago and Weber in order to be able to construct a larger class of quantum groups. The easy quantum groups arising from their definition of two-colored partitions sit between $S_n$ and the free unitary quantum group $U_n^+$ defined also by Wang \cite{wang1995free}.

The categories of partitions according to the original definition of Banica and Speicher were already fully classified \cite{raum2013full} providing the full classification of \emph{orthogonal easy quantum groups}. The paper \cite{tarragowebercombina} provides a classification of two-colored partitions only in the so called \emph{non-crossing} case and the \emph{group} case. Thus, the problem of classification of \emph{unitary easy quantum groups} is still open.

In this paper we deal with a special case when the partition $\paarpartww\otimes\paarpartbb$ is an element of the category. This case is called \emph{globally colorized} and in some sense it behaves very similarly as the categories with no colors. Thanks to this property we are able to find a complete classification of such categories.

\begin{customthm}{A}[Theorem \ref{Th}]\label{ThA}
	Every globally colorized category $\CC$ is determined by the number $k(\CC)$ and a non-colored category of partitions not containing the singleton. The full classification of globally colorized categories is summarized in Table~\ref{t} (see page \pageref{t}).
\end{customthm}

In the second part of the article we prove that the corresponding unitary easy quantum groups can be constructed from the orthogonal ones by a process called \emph{tensor complexification}. This construction was defined in \cite{tarragoweberopalg}, where it was applied on the non-crossing globally colorized categories. Using the same proof technique, we are able to formulate this theorem for any globally colorized category.

\begin{customthm}{B}[Theorem \ref{Th2}, Corollary \ref{C.glue}]\label{ThB}
Let $\CC$ be a globally colorized category of partitions, denote $k:=k(\CC)$. Denote $H\subset O_n^+$ the quantum group corresponding to the non-colored category of partitions $\langle\CC,\paarpartww\rangle$. Then $\CC$ corresponds to the quantum group $H\tensorglued\hat\Z_k\subset U_n^+$. In the case $k=0$, we replace $\hat\Z_k$ by $\hat\Z$. Conversely, any quantum group of the form $H\tensorglued\hat\Z_k$, where $H$ is an orthogonal easy quantum group and $k\in\N_0$, is a unitary easy quantum group corresponding to a globally colorized category.
\end{customthm}


The paper is structured as follows. The first three sections belong to the purely combinatorial part of the paper. In the first section we recall the basic definitions for categories of partitions. In Section \ref{sec.noncol} we summarize the classification results for categories of non-colored partitions and we also mention important results from \cite{tarragowebercombina} regarding the relationship between categories of partitions with colors and those without colors. Finally, Section \ref{sec.class} contains the classification of globally colorized categories of partitions.

The quantum group part of the article is introduced in Section \ref{sec.unitary} summarizing the definition of unitary easy quantum groups. In Section \ref{sec.globqg} we recall the tensor complexification and prove Theorem \ref{ThB}.

\section*{Acknowledgements}

The author was supported by the collaborative research centre SFB-TRR 195 ``Symbolic Tools in Mathematics and their Application''. The author is grateful to his PhD advisor prof. Moritz Weber for numerous suggestions and comments.

\section{Categories of partitions}

First of all we recall the definition of two-colored partitions, which was introduced in \cite{tarragowebercombina} and which extends the combinatorial approach to quantum groups introduced by Banica and Speicher in \cite{banica2009liberation}.

\subsection{Partitions}

A \emph{non-colored partition} of $k$ points, $k\in\N_0$ is a partition of the set $\{1,2,\dots,k\}$, that is, a decomposition of the set $\{1,\dots,k\}$ into non-empty disjoint subsets called \emph{blocks}. The set of all non-colored partitions of $k$ points is denoted by $P(k)$. The set of all partitions is denoted by $P=\bigcup_{k=0}^\infty P(k)$. The number $k$ is called the \emph{length} of the partition and is denoted by $|p|=k$ for $p\in P(k)$.

We illustrate partitions graphically by putting $k$ points representing the numbers 1, \dots, $k$ in a row and connecting by lines those that are grouped in one block. All lines are drawn above the points. Another possibility of representing the partitions is to assign a letter to each block. Then a partition $p\in P(k)$ can be represented by the word $a_1a_2\cdots a_k$, where $a_i$ is the letter representing the block of the point $i$.

Below, we give an example of two partitions $p,q\in P(7)$ including their graphical and word representation. The first set of points is decomposed into three blocks, whereas the second one has four blocks. In addition, the first one is an example of a \emph{non-crossing} partition, i.e. a partition that can be drawn in a way that lines connecting different blocks do not intersect (following the rule that all lines are above the points). On the other hand, the second partition has one crossing.

\begin{eqnarray*}
\symstrut 0.7cm 
p&:=&\{\{1,2,7\},\{3,4\},\{5,6\}\}=
\begin{picture}(7,1)(0,0.5)
\put(-1,0){\uppartiii{2}{1}{2}{7}}
\put(-1,0){\uppartii{1}{3}{4}}
\put(-1,0){\uppartii{1}{5}{6}}
\end{picture}
=\mathsf{aabbcca}
\\
q&:=&\{\{1,2\},\{3,5\},\{4,6\},\{7\}\}=
\begin{picture}(7,1)(0,0.5)
\put(-1,0){\uppartii{1}{1}{2}}
\put(-1,0){\uppartii{1}{3}{5}}
\put(-1,0){\uppartii{2}{4}{6}}
\put(-1,0){\upparti{1}{7}}
\end{picture}
=\mathsf{aabcbcd}
\end{eqnarray*}

\emph{A two-colored partition} is a generalization of this concept, where, in addition, we assign to each point the color white or black. Bellow, we show an example of a possible coloring of the partitions above. The set of all two-colored partitions of $k$ points is denoted by $P^{\twocol}(k)$ and we also denote $P^{\twocol}=\bigcup_k P^{\twocol}(k)$.

$$
\symstrut 0.7cm 
p=
\begin{picture}(7,1)(0,0.5)
\put(-1,0){\uppartiii{2}{1}{2}{7}}
\put(-1,0){\uppartii{1}{3}{4}}
\put(-1,0){\uppartii{1}{5}{6}}
\put(0.05,-0.35){$\bullet$}
\put(1.05,-0.35){$\circ$}
\put(2.05,-0.35){$\bullet$}
\put(3.05,-0.35){$\bullet$}
\put(4.05,-0.35){$\bullet$}
\put(5.05,-0.35){$\circ$}
\put(6.05,-0.35){$\bullet$}
\end{picture}
\qquad
q=
\begin{picture}(7,1)(0,0.5)
\put(-1,0){\uppartii{1}{1}{2}}
\put(-1,0){\uppartii{1}{3}{5}}
\put(-1,0){\uppartii{2}{4}{6}}
\put(-1,0){\upparti{1}{7}}
\put(0.05,-0.35){$\bullet$}
\put(1.05,-0.35){$\circ$}
\put(2.05,-0.35){$\circ$}
\put(3.05,-0.35){$\bullet$}
\put(4.05,-0.35){$\bullet$}
\put(5.05,-0.35){$\circ$}
\put(6.05,-0.35){$\circ$}
\end{picture}
$$

In the colored case we can also use the word representation, but have to somehow indicate the colors. We may, for example, use lower case letters for white and upper case letters for black, so the example looks as follows.
$$p=\mathsf{AaBBCcA},\qquad q=\mathsf{AabCBcd}.$$

In this paper, by {\em partition}, we mean a two-colored partition if not stated otherwise. We will use mostly the graphical representation to describe partitions since it is more illustrative. However, in some cases, it will be more convenient to use the word representation.

\subsection{Operations on partitions}\label{sec.operations}

We define the following operations on the set $P^\twocol$. Each operation is provided with an example using the partitions $p$ and $q$ defined above.
\begin{itemize}
 \item  The \emph{tensor product} of two partitions $p\in P^{\twocol}(k)$ and $q\in P^{\twocol}(k')$ is the partition $p\otimes q\in P^{\twocol}(k+k')$ obtained by writing the graphical representations of $p$ and $q$ ``side by side''
$$
\symstrut 0.7cm 
p\otimes q=
\begin{picture}(14,1)(0,0.5)
\put(-1,0){\uppartiii{2}{1}{2}{7}}
\put(-1,0){\uppartii{1}{3}{4}}
\put(-1,0){\uppartii{1}{5}{6}}
\put(0.05,-0.35){$\bullet$}
\put(1.05,-0.35){$\circ$}
\put(2.05,-0.35){$\bullet$}
\put(3.05,-0.35){$\bullet$}
\put(4.05,-0.35){$\bullet$}
\put(5.05,-0.35){$\circ$}
\put(6.05,-0.35){$\bullet$}
\put(-1,0){\uppartii{1}{8}{9}}
\put(-1,0){\uppartii{1}{10}{12}}
\put(-1,0){\uppartii{2}{11}{13}}
\put(-1,0){\upparti{1}{14}}
\put(7.05,-0.35){$\bullet$}
\put(8.05,-0.35){$\circ$}
\put(9.05,-0.35){$\circ$}
\put(10.05,-0.35){$\bullet$}
\put(11.05,-0.35){$\bullet$}
\put(12.05,-0.35){$\circ$}
\put(13.05,-0.35){$\circ$}
\end{picture}
$$
 \item Let $p\in P^\twocol(k)$ be a partition such that the $i$-th point has a different color than the $(i+1)$-st point for some $i\in\{1,\dots,k\}$ (if $i=k$, the $(i+1)$-st point means the first point). The \emph{contraction} of $p$ at $i$-th position is the partition  $\Pi_ip\in P^{\twocol}(k-2)$ obtained by unification of the blocks of the $i$-th and $(i+1)$-st point and removing those points.
$$
\symstrut 0.7cm 
\Pi_2p=
\begin{picture}(7,1)(0,0.5)
\put(-1,0){\uppartiii{2}{1}{2}{7}}
\put(-1,0){\uppartii{1}{3}{4}}
\put(-1,0){\uppartii{1}{5}{6}}
\put(-1,1){\partii{1}{2}{3}}
\put(0.05,-0.35){$\bullet$}
\put(3.05,-0.35){$\bullet$}
\put(4.05,-0.35){$\bullet$}
\put(5.05,-0.35){$\circ$}
\put(6.05,-0.35){$\bullet$}
\end{picture}
=
\begin{picture}(5,1)(0,0.5)
\put(-1,0){\uppartiii{2}{1}{2}{5}}
\put(-1,0){\uppartii{1}{3}{4}}
\put(0.05,-0.35){$\bullet$}
\put(1.05,-0.35){$\bullet$}
\put(2.05,-0.35){$\bullet$}
\put(3.05,-0.35){$\circ$}
\put(4.05,-0.35){$\bullet$}
\end{picture}
$$
$$
\symstrut 0.7cm 
\Pi_7q=
\begin{picture}(7,1)(0,0.5)
\put(-1,0){\uppartii{1}{1}{2}}
\put(-1,0){\uppartii{1}{3}{5}}
\put(-1,0){\uppartii{2}{4}{6}}
\put(-1,0){\upparti{1}{7}}
\put(-1,1){\partii{1}{1}{7}}
\put(1.05,-0.35){$\circ$}
\put(2.05,-0.35){$\circ$}
\put(3.05,-0.35){$\bullet$}
\put(4.05,-0.35){$\bullet$}
\put(5.05,-0.35){$\circ$}
\end{picture}
=
\begin{picture}(5,1)(0,0.5)
\put(-1,0){\upparti{1}{1}}
\put(-1,0){\uppartii{1}{2}{4}}
\put(-1,0){\uppartii{2}{3}{5}}
\put(0.05,-0.35){$\circ$}
\put(1.05,-0.35){$\circ$}
\put(2.05,-0.35){$\bullet$}
\put(3.05,-0.35){$\bullet$}
\put(4.05,-0.35){$\circ$}
\end{picture}
$$
 \item The \emph{reflection} of a partition $p\in P^{\twocol}(k)$ is the partition $\tilde p\in P^{\twocol}(k)$ obtained by reflecting $p$ at the vertical axis and inverting all colors of the points.
$$
\symstrut 0.7cm 
\tilde p=
\begin{picture}(7,1)(0,0.5)
\put(-1,0){\uppartiii{2}{1}{6}{7}}
\put(-1,0){\uppartii{1}{2}{3}}
\put(-1,0){\uppartii{1}{4}{5}}
\put(0.05,-0.35){$\circ$}
\put(1.05,-0.35){$\bullet$}
\put(2.05,-0.35){$\circ$}
\put(3.05,-0.35){$\circ$}
\put(4.05,-0.35){$\circ$}
\put(5.05,-0.35){$\bullet$}
\put(6.05,-0.35){$\circ$}
\end{picture}
$$
 \item The \emph{rotation} of a partition $p\in P^\twocol(k)$ is the partition $Rp\in P^\twocol(k)$ given by shifting the last point to the beginning not changing the block it belongs to.
$$
\symstrut 0.5cm 
Rp=
\begin{picture}(7,1)(0,0.5)
\put(-1,0){\uppartiii{1}{1}{2}{3}}
\put(-1,0){\uppartii{1}{4}{5}}
\put(-1,0){\uppartii{1}{6}{7}}
\put(0.05,-0.35){$\bullet$}
\put(1.05,-0.35){$\bullet$}
\put(2.05,-0.35){$\circ$}
\put(3.05,-0.35){$\bullet$}
\put(4.05,-0.35){$\bullet$}
\put(5.05,-0.35){$\bullet$}
\put(6.05,-0.35){$\circ$}
\end{picture}
$$
\end{itemize}

These operations are called the \emph{category operations} on partitions.

The category operations on non-colored partitions are defined in the same way as in the colored case ignoring any information about colors. That is, we do not have to change colors when reflecting vertically and we do not have to check the colorings for performing contractions.
%

\subsection{Categories of partitions} \label{sec.category}

A set of partitions $\CC\subset P^{\twocol}$ is a \emph{category of partitions}, if it is closed under the category operations and if it contains the \emph{two-colored pair partitions} $\paarpartwb,\paarpartbw\in P^\twocol(2)$. We denote by $\CC(k):=\{p\in\CC\mid |p|=k\}$.

We denote $\CC=\langle p_1,\dots,p_n\rangle$ the smallest category containing given partitions $p_1,\dots,p_n\in P^\twocol$. We say that the category $\CC$ is \emph{generated} by those partitions. Note that the pair partitions are contained in such category by definition, so we do not explicitly write those generators.

\subsection{Some important partitions}\label{ssec.important}

We already mentioned the definition of pair partitions in the definition of category of partitions. Note that appart from the \emph{two-colored} pair partitions $\paarpartwb$, $\paarpartbw$, we have also \emph{unicolored} pair partitions $\paarpartww$ and $\paarpartbb$. 

Then, let us mention the partition consisting of a single point, which will be called the \emph{singleton} and for clarity denoted by arrow as $\singletonw$, $\singletonb$, or $\singleton$ (the white, black, and non-colored version).

A partition, where all points belong to a single block will be called simply a \emph{block partition}. The block of $k$ white points will be denoted $b_k$ and the block of $k$ black points will be denoted $b_{-k}$. So, for example $b_2=\paarpartww$, $b_{-3}=\dreipartbbb$.

Let us also define the following
$$u_k:=\paarpartww^{\otimes k/2},\quad u_{-k}:=\tilde u_k=\paarpartbb^{\otimes k/2}\quad \hbox{for $k\in2\N_0=\{0,2,4,\dots\}$},$$
$$s_k:=\singletonw^{\otimes k},\quad s_{-k}:=\tilde s_k=\singletonb^{\otimes k}\quad \hbox{for $k\in\N_0$}.$$

Some partitions are particularly important, because their presence in some category tells us that the category is closed under some additional operations. In this paper, we are especially interested in operations of shifting or changing colors provided by partitions $\paarpartww$ and $\paarpartww\otimes\paarpartbb$. Those properties will be studied in Section~\ref{sec.noncol}. Other important partitions are $\singletonw\otimes\singletonb$ and $\vierpartwbwb$. They allow us to do the following operations.

\begin{lem}[{\cite[Lemma 1.3]{tarragowebercombina}}]\label{L.imppart}
Let $\CC$ be a category of partitions.
\begin{itemize}
\item If $\singletonw\otimes\singletonb\in\CC$, then $\CC$ is closed under ``disconnecting points''. That is, if $p\in\CC$, then $p'\in\CC$, where $p'$ was created from $p$ by disconnecting one point and thus making a new block consisting of this point.
\item If $\vierpartwbwb\in\CC$, then $\CC$ is closed under connecting neighbouring blocks if they meet at two points with inverse colors.
\end{itemize}
\end{lem}

We can expect that the classification of categories of partitions could be quite different depending on whether those partitions are elements or not. Therefore, we define the following four cases for categories of partitions.
\begin{itemize}
\item $\CC$ is in \emph{case $\OOO$}, if $\singletonw\otimes\singletonb\notin\CC$ and $\vierpartwbwb\notin\CC$,
\item $\CC$ is in \emph{case $\BBB$}, if $\singletonw\otimes\singletonb\in\CC$ and $\vierpartwbwb\notin\CC$,
\item $\CC$ is in \emph{case $\HHH$}, if $\singletonw\otimes\singletonb\notin\CC$ and $\vierpartwbwb\in\CC$,
\item $\CC$ is in \emph{case $\SSS$}, if $\singletonw\otimes\singletonb\in\CC$ and $\vierpartwbwb\in\CC$.
\end{itemize}
The letters $\OOO$, $\BBB$, $\HHH$, and $\SSS$ stand for orthogonal, bistochastic, hyperoctahedral, and symmetric group, which are groups corresponding to particular instances of those categories.

\subsection{Partitions with upper and lower points}\label{subsec.upper}

The usual definition of partitions used in papers related to easy quantum groups (including \cite{banica2009liberation,tarragowebercombina}) is a bit different. To each partition $p\in P^\twocol(k)$ (or $p\in P(k)$) we, in addition, assign a number $l\in\{0,\dots,k\}$ and we say that the first $k-l$ points are \emph{upper} and the remaining $l$ points are \emph{lower}. The category operations are defined a bit differently giving rise to different definition of the category of partitions.

\begin{prop} Our definition is equivalent to the original definition in the following sense.
\begin{itemize}
	\item Let $\tilde\CC$ be a category of partitions according to \cite{banica2009liberation,tarragowebercombina}. Let $\CC\subset P^{\twocol}$ be the set of partitions in $\tilde\CC$ containing lower points only (i.e. $l=k$). Then $\CC$ is a category of partitions in the sense of Section \ref{sec.category}.
	\item Let $\CC\subset P^{\twocol}$ be a category of partitions in the sense of Section \ref{sec.category}. Let $\tilde\CC$ be the set of all partitions $p\in\CC$ together with all possible choices of $l$. Then $\tilde\CC$ is a category of partitions in the sense of \cite{banica2009liberation,tarragowebercombina}.
\end{itemize}
This correspondence is bijective.
\end{prop}
\begin{proof}
	Let $\tilde\CC$ be a category of partitions according to \cite{banica2009liberation,tarragowebercombina}. By definition it contains the pair partitions $\paarpartwb$ and $\paarpartbw$. As follows from \cite[Lemma 1.1]{tarragowebercombina}, the subset of partitions with lower points is closed under the category operations as defined in Section \ref{sec.operations}. Thus $\CC$ is a category of partitions according to Section \ref{sec.category}.
	
	In addition, \cite[Lemma 1.1, item (a)]{tarragowebercombina} essentially says that the category $\tilde\CC$ is closed under changing the number of upper/lower points. It follows that $\tilde\CC$ can be reconstructed from $\CC$ by considering all partitions $p\in\CC$ together with all possible choices of the number $l$.
%
\end{proof}

%
%
%

\section{Globally colorized partitions}\label{sec.noncol}

The aim of this article is to classify globally-colorized categories of two-colored partitions, which are in some sense very close to the categories of partitions without colors. The result is based on the available classification of non-colored categories of partitions.

In this section we summarize the results on the classification of categories of non-colored partitions, define the globally-colorized categories and recall important results about the relationship between the non-colored and two-colored categories of partitions obtained in \cite{tarragowebercombina}.

\subsection{Classification in the case of non-colored partitions}\label{ssec.classnocol}

The categories of non-colored partitions were studied in \cite{banica2009liberation,banicacurranspeicher2010,weber2013classification,raum2014combinatorics,raum2013easy} and their full classification was completed in \cite{raum2013full}.

The classification is summarized in the following table. It is divided into four cases. The \emph{non-crossing categories} consisting of non-crossing partitions, the \emph{group categories} containing the \emph{crossing partition} $\crosspartrot$, the \emph{half-liberated categories} containing the \emph{half-liberating partition} $\halflibpartrot$, but not the crossing partition, and the rest.

\medskip\noindent
\begin{tabular}{lll}
Non-crossing           & $\langle\rangle$, $\langle\vierpart\rangle$, $\langle\vierpart,\singleton\otimes\singleton\rangle$, $\langle\singleton\otimes\singleton\rangle$, $\langle\legpart\rangle$, $\langle\singleton\rangle$, $\langle\singleton,\vierpart\rangle$& \\
Group                  & $\langle\crosspartrot\rangle$, $\langle\crosspartrot,\singleton\rangle$, $\langle\crosspartrot,\singleton\otimes\singleton\rangle$&\\
		       & $\langle\crosspartrot,\vierpart\rangle$, $\langle\crosspartrot,\vierpart,\singleton\rangle$, $\langle\crosspartrot,\vierpart,\singleton\otimes\singleton\rangle$&($*$)\\
Half-liberated         & $\langle\halflibpartrot\rangle$, $\langle\halflibpartrot\singleton\otimes\singleton\rangle$&\\
		       & $\langle\halflibpartrot,\vierpart\rangle$, $\langle\halflibpartrot,\vierpart,h_s\rangle, s\ge 3$&($*$)\\
The rest               & $\langle\pi_s\rangle, s\ge 2$, $\langle\pi_l\mid l\in\N\rangle$, $A\unlhd\Z_2^\infty$&\\
\end{tabular}
\medskip

Here, we denote by $h_s$, $s\in\N_0$ (here $s\ge 3$) the partition of length $2s$ consisting of two blocks, where the first block connects all odd points and the second block connects all even points. We can represent this partition by the word $(\mathsf{ab})^s=\underbrace{\mathsf{ab\,ab\,ab\cdots ab}}_{s\times}$.

By $\pi_s$, we denote the partition represented by the following word of length $4s$
$$\pi_s=\la_1\la_2\cdots \la_s\la_s\cdots \la_2\la_1\la_1\la_2\cdots \la_s\la_s\cdots \la_2\la_1.$$

Finally, consider the infinite free product $\Z_2^{*\infty}$ and denote its generators by $\la_i$, $i=1,2,\dots$ Then the elements of $\Z_2^{*\infty}$ are words over the alphabet $\{\la_i\}_{i\in\N}$ and therefore stand for partitions. It can be shown that if a normal subgroup $A\unlhd\Z_2^{*\infty}$ is so called $\mathrm{sS_\infty}$ invariant, i.e. invariant with respect to the homomorphisms of the form $\la_i\mapsto\la_{\phi(i)}$ where $\phi\colon\N\to\N$ is an arbitrary map, then the corresponding partitions form a category. Such categories are called \emph{group-theoretical}. This is indicated in the last entry of the summary.

The group categories and half-liberated categories in the rows marked by asterisk ($*$) are special instances of group-theoretical categories. Except for this fact, all the categories in the summary are pairwise distinct.

\subsection{Non-colored categories of colored partitions}\label{sec.psi}

As described in \cite{tarragowebercombina}, there is a bijection between categories of non-colored partitions and categories of two-colored partitions containing the unicolored pair~\paarpartww. More precisely, the following holds. 

\begin{lem}[\cite{tarragowebercombina}]
Let $\CC$ be a category of partitions containing the unicolored pair partition $\paarpartww\in\CC$. Then $\CC$ is closed under arbitrary choice of colors. That is, if $p\in\CC$, then $p'\in\CC$, where $p'$ is obtained from $p$ by making arbitrary choice for the colors of the points (keeping all blocks the same).
\end{lem}

This means that if a category contains the unicolored pair, then the coloring of its elements is irrelevant. Hence, such a category can be identified with a category of non-colored partitions in the following way.

Let $\Psi\colon P^{\twocol}\to P$ be the map given by forgetting the colors of a two-colored partition. For $\CC\subset P$, denote $\Psi^{-1}(\CC)\subset P^{\twocol}$ its preimage under $\Psi$.

\begin{prop}[{\cite[Proposition 1.4]{tarragowebercombina}}]\label{PropOneColored}
\par\noindent
\begin{itemize}
\item[(a)] Let $\CC\subset P$ be a category of non-colored partitions. Then $\Psi^{-1}(\CC)\subset P^{\twocol}$ is a category of two-colored partitions containing the unicolored pair partitions $\paarpartww$ and $\paarpartbb$.
\item[(b)] Let $\CC\subset P^{\twocol}$ be a category of two-colored partitions containing the unicolored pair partition $\paarpartww$ (or equivalently $\paarpartbb$). Then $\Psi(\CC)\subset P$ is a category of non-colored partitions and $\Psi^{-1}(\Psi(\CC))=\CC$.
\end{itemize}
Hence, there is a one-to-one correspondence between categories of non-colored partitions and categories of two-colored partitions containing $\paarpartww$.
\end{prop}

In the following, we will not distinguish between categories of non-colored partitions and categories of two-colored partitions containing \paarpartww. In particular, a (two-colored) category will be called \emph{non-colored} if it contains \paarpartww.

\subsection{Globally colorized categories} The tensor product $\paarpartww\otimes\paarpartbb$ possesses a similar, but weaker property.

\begin{lem}[\cite{tarragowebercombina}]\label{L.colperm}
Let $\CC$ be a category of partitions such that $\paarpartww\otimes\paarpartbb\in\CC$. Then $\CC$ is closed under permutation of colors. That is, if $p\in\CC$, then $p'\in\CC$, where $p'$ is obtained from $p$ by changing colors in such a way that the number of white points (and black points) in $p'$ is the same as in $p$.
\end{lem}

So, in the case of globally colorized categories the only thing that matters is the number of white and black points. The actual distribution of colors in a partition is irrelevant. This was also pointed out in \cite{tarragowebercombina}, where the following two definitions were introduced.

\begin{defn}\label{DefGlobalColor}
A category of partitions $\CC\subset P^{\twocol}$ is 
\begin{itemize}
 \item \emph{globally colorized}, if $\paarpartww\otimes\paarpartbb\in\CC$
 \item and \emph{locally colorized} if $\paarpartww\otimes\paarpartbb\notin\CC$.
\end{itemize}
\end{defn}

\begin{defn}
For a partition $p\in P^{\twocol}$ we denote by $c_\circ(p)$ the number of white points in $p$, $c_\bullet(p)$ the number of black points in $p$ and finally we define $c(p):=c_\circ(p)-c_\bullet(p)$.
\end{defn}

\begin{defn}
Let $\CC$ be a category of  partitions. We set $k(\CC)$ as the minimum  of all numbers $c(p)$ such that $c(p)>0$ and $p\in\CC$, if such a  partition exists in $\CC$. Otherwise $k(\CC):=0$. 
The parameter $k(\CC)$ is called the \emph{degree of reflection} of $\CC$. It is the \emph{global parameter} of $\CC$.
\end{defn}

\begin{lem}[{\cite[Lemma 2.6]{tarragowebercombina}}]\label{L.cprops}
For the map $c\colon P^\twocol\to\Z$, the following holds true.
\begin{enumerate}
\item $c(p\otimes q)=c(p)+c(q)$,
\item $c(\Pi_i p)=c(p)$,
\item $c(\tilde p)=-c(p)$,
\item $c(Rp)=c(p)$.
\end{enumerate}
Here, we suppose that the assumptions for the contraction in (2) are satisfied.
\end{lem}

\begin{lem}[{\cite[Proposition 2.7]{tarragowebercombina}}]\label{L.kdivc}
	Let $\CC$ be a category of partitions. If the number $k(\CC)$ is nonzero, then it is a divisor of $c(p)$ for every $p\in P^\twocol$.
\end{lem}


\section{The classification of globally colorized categories}\label{sec.class}

In this section we perform the classification of globally colorized categories.

\subsection{Categories with zero degree of reflection}

As we described in the previous section, the globally colorized categories are somehow close to the non-colored ones. Such a correspondence is a central tool for their classification. In this subsection we formulate this correspondence in the case when $k(\CC)=0$.

\begin{lem}\label{L.singleton}
Let $\CC$ be a category of partitions.
\begin{enumerate}
\item If $k(\CC)=0$ then all partitions in $\CC$ have even length.
\item Let $\paarpartww\in\CC$. Then all partitions in $\CC$ have even length if and only if $\singletonw\not\in\CC$.
\end{enumerate}
\end{lem}
\begin{proof}
The first part is obvious: If a partition $p\in\CC$ has an odd length then it cannot have the same amount of white and black points, so $c(p)\neq 0$ and hence $k(\CC)\neq 0$.

The assumption of (2) means that $\CC$ is a non-colored category. Let $p\in\CC$ have odd length. Since the coloring of $p$ does not matter, we can perform contractions (and necessary color changes) on $p$ until we get the singleton $\singletonw$.
\end{proof}

\begin{defn}\label{d.0}
Let $\CC$ be a category of partitions. Denote $\CC_0:=\{p\in\CC\mid c(p)=0\}$.
\end{defn}

\begin{lem}
Let $\CC$ be a category of partitions. Then $\CC_0$ is a category of partitions with $k(\CC_0)=0$.
\end{lem}
\begin{proof}
	Using Lemma \ref{L.cprops} we see that the category operations applied on partitions with $c(p)=0$ again produces partitions with $c(p)=0$. Thus, the subset $\CC_0\subset\CC$ is closed under the category operations. Moreover, $c(\paarpartwb)=c(\paarpartbw)=0$, so $\paarpartwb,\paarpartbw\in\CC_0$.

The fact that $k(\CC_0)=0$ follows directly from the definition of $\CC_0$.
\end{proof}

Recall the definition of the map $\Psi\colon P^\twocol\to P$ from Section \ref{sec.psi}. Then we can formulate the following.

\begin{lem}\label{l.bij1}
Let $\CC$ be a globally colorized category with $k(\CC)=0$ (or equivalently $\CC=\CC_0$). Then
\begin{enumerate}
\item $\Psi(\CC)$ is a category of non-colored partitions. 
\item In terms of two-colored partitions, it corresponds to the non-colored category $\Psi^{-1}(\Psi(\CC))=\langle\CC,\paarpartww\rangle$.
\item Moreover, $\CC=\langle\CC,\paarpartww\rangle_0$.
\end{enumerate}
\end{lem}
\begin{proof}
	To prove the first statement, we have to show that $\Psi(\CC)$ contains pair partitions and it is closed under category operations. We have that $\paarpartwb\in\CC$, so $\paarpart\in\Psi(\CC)$. For $p,q\in\CC$ we have that $\Psi(p)\otimes\Psi(q)=\Psi(p\otimes q)$. We also have that $\widetilde{\Psi(p)}=\Psi(\tilde p)$. Finally, for any $p\in\CC(k)$ and any $i\in\{1,\dots,k\}$, we can shift colors in $p$ in a way that the $i$-th and $(i+1)$-st point have different colors. Then $\Pi_i(\Psi(p))=\Psi(\Pi_i p)$.

	The second statement is obvious. Surely both $\CC\subset\Psi^{-1}(\Psi(C))$ and $\paarpartww\in\Psi^{-1}(\Psi(\CC))$, so we have the inclusion $\supset$. For the other one, we use Proposition \ref{PropOneColored} to see that $\Psi^{-1}(\Psi(\CC))\subset\Psi^{-1}(\Psi(\langle\CC,\paarpartww\rangle))=\langle\CC,\paarpartww\rangle$.

	In the third statement, the inclusion $\subset$ is obvious. For the opposite inclusion, take $p\in\langle\CC,\paarpartww\rangle_0$. This means that we are taking an arbitrary element of $\langle\CC,\paarpartww\rangle$ such that $c(p)=0$. Applying $\Psi$ on the equality in (2), we get $\Psi(\CC)=\Psi(\langle\CC,\paarpartww\rangle)$. This means that there is $p'\in\CC$ such that $\Psi(p')=\Psi(p)$. Since we have $c(p')=c(p)=0$ and since $\CC$ is globally colorized, we can shift colors in $p'$ to obtain $p\in\CC$.
\end{proof}

\begin{lem}\label{l.bij2}
Let $\tilde\CC\subset P$ be a category of non-colored partitions such that $\singleton\not\in\tilde\CC$. Denote $\bar\CC:=\Psi^{-1}(\tilde\CC)\subset P^{\twocol}$ the corresponding category in terms of two-colored partitions.
\begin{enumerate}
\item It holds that $\tilde\CC=\Psi(\bar\CC_0)$ or, equivalently, $\bar\CC=\langle\bar\CC_0,\paarpartww\rangle$.
\item If $\tilde\CC=\langle p_1,\dots,p_n\rangle$ for some $p_1,\dots,p_n\in P$, then $\bar\CC_0=\langle p_1',\dots,p_n',\paarpartww\otimes\paarpartbb\rangle$, where, for every $i\in\{1,\dots,n\}$, $p_i'$ is a coloring of $p_i$ such that $c(p_i')=0$.
\end{enumerate}
\end{lem}
\begin{proof}
	The inclusion $\tilde\CC\supset\Psi(\bar\CC_0)$ is obvious. For the opposite, take $p\in\tilde\CC$. Since $\singleton\not\in\CC$, $p$ must have even length. Thus, it has a coloring $p'$ such that $c(p')=0$, so $p'\in\bar\CC_0$ and hence $p\in\Psi(\bar\CC_0)$. The equivalent description of this equality follows from Lemma \ref{l.bij1}.

For the second statement, the existence of appropriate $p_1',\dots,p_n'$ again follows from the fact that $\singleton\not\in\CC$, so all the partitions $p_1,\dots,p_n$ have even length. Then it is easy to see that all the generators $p_1',\dots,p_n'$ and $\paarpartww\otimes\paarpartbb$ are elements of $\bar\CC_0$. Finally, take arbitrary $p\in\bar\CC_0$. Then
$$\Psi(p)\in\Psi(\bar\CC_0)=\tilde\CC=\langle p_1,\dots,p_n\rangle\subset\Psi(\langle p_1',\dots,p_n',\paarpartww\otimes\paarpartbb\rangle),$$
	where we used that $\langle p_1',\dots,p_n',\paarpartww\otimes\paarpartbb\rangle$ satisfies the assumptions of Lemma \ref{l.bij1}, so its image under $\Psi$ is a category containing $p_1,\dots,p_n$. So, there is $p'\in\langle p_1',\dots,p_n',\paarpartww\otimes\paarpartbb\rangle$ such that $\Psi(p)=\Psi(p')$. Since $c(p)=c(p')=0$ and since the category $\langle p_1',\dots,p_n',\paarpartww\otimes\paarpartbb\rangle$ is globally colorized, it must contain also $p$.
\end{proof}

\begin{prop}\label{p.bij}
There is a bijection between globally colorized categories $\CC$ with $k(\CC)=0$ and non-colored categories $\bar\CC$ with $\singletonw\not\in\bar\CC$ given by $\CC\mapsto\langle\CC,\paarpartww\rangle$ with inverse $\bar\CC\mapsto\bar\CC_0$.
\end{prop}
\begin{proof}
	Denote $\Phi_1$ the mentioned map and $\Phi_2$ the alleged inverse. Lemma \ref{l.bij1} says that $\Phi_2\Phi_1$ is the identity and Lemma \ref{l.bij2} says that $\Phi_1\Phi_2$ is the identity.
\end{proof}

\subsection{Categories with non-zero degrees of reflection}

\begin{lem}
\label{l.even}
Let $\CC$ be a globally colorized category.
\begin{enumerate}
\item If $\singletonw\otimes\singletonb\not\in\CC$, then $\positionerwwbb\not\in\CC$.
\item If $\positionerwwbb\not\in\CC$, then all partitions have even length.
\end{enumerate}
\end{lem}
\begin{proof}
The first part follows from the fact that $\singletonw\otimes\singletonb$ is a contraction of $\positionerwwbb$.

Now, suppose that $|p|$ is odd for $p\in\CC$. Then take
$$p':=p\otimes R(\tilde p\otimes\paarpartbw)
=p
\begin{picture}(3,1)(0,0.5)
\put(-1,0.5){\uppartii{1}{1}{3}}
\put(0.05,0.15){$\circ$}
\put(1.05,0.5){$\tilde p$}
\put(2.05,0.15){$\bullet$}
\end{picture}.
$$
Since there is the same amount of black and white points in $p'$, we can use Lemma \ref{L.colperm} to permute colors in such a way that we can contract the partition to $\positionerwwbb$.
\end{proof}

Recall the definition of partitions $u_k$ and $s_k$ in Section \ref{ssec.important}.

\begin{lem}\label{l}
Let $\CC$ be a globally colorized category, denote $k:=k(\CC)$.
\begin{enumerate}
\item If $\CC$ contains only partitions of even length (in particular, if $\positionerwwbb\not\in\CC$ or $\singletonw\otimes\singletonb\not\in\CC$), then $k$ is even and $\CC=\langle\CC_0,u_k\rangle$.
\item If $\singletonw\otimes\singletonb\in\CC$, then $\CC=\langle\CC_0,s_k\rangle$.
\end{enumerate}
\end{lem}
\begin{proof}
We begin with the statement of (1). Suppose that every partition $p\in\CC$ has even length. Since $k(\CC)=k$, there is a partition $p\in\CC$ such that $c(p)=k$. We can repeatedly perform contractions on $p$ until we get $p'$ containing only $k$ white points. Since every partition is of even length, we have that $k$ is even.

Now, let us prove the inclusion $\supseteq$ in (1). Again, take $p\in\CC$ such that $c(p)=k$. Using Lemma \ref{L.colperm} we can permute colors in $\paarpartwb^{\otimes k/2}\otimes p\in\CC$ in such a way that the first $k$ points are white. Then we can perform contractions on the rest of the points and obtain $\paarpartww^{\otimes k/2}=u_k$. Note that $u_k\in\CC$ also implies that $u_{nk}\in\CC$ for any $n\in\Z$. Indeed, for positive $n$, we have $u_{nk}=u_k^{\otimes n}$ and $u_{-nk}=\tilde u_k^{\otimes n}$.

Now, we prove the inclusion $\supseteq$ in (2). Again, take $p\in\CC$ such that $c(p)=k$. Perform contractions until we get $p'$ consisting of $k$ white points only. Now, according to Lemma \ref{L.imppart} we can disconnect all points in $p'$ using $\singletonw\otimes\singletonb$ to obtain $\singletonw^{\otimes k}=s_k$. Again, note that it follows that $s_{nk}\in\CC$ for any $n\in\Z$.

The proof of the inclusion $\subseteq$ is the same in both cases. Denote $q_k$ the partition $u_k$ or $s_k$ depending on whether we are in the case (1) or (2). We have already proven that $q_{nk}\in\CC$ for all $n\in\Z$. Take an arbitrary $p\in\CC$. From Lemma \ref{L.kdivc} we see that $c(p)$ is a multiple of $k$, so $q_{-c(p)}\in\CC$ and hence $p':=p\otimes q_{-c(p)}\in\CC$. Since $c(p')=0$, we have $p'\in\CC_0$. Finally, $p$ can be obtained by repeated contraction of $p\otimes q_{-c(p)}\otimes q_{c(p)}=p'\otimes q_{c(p)}\in\langle\CC_0,q_k\rangle$.
\end{proof}

\begin{lem}\label{l.eq}
	Let $\CC_1, \CC_2$ be non-colored categories such that $\singletonw\not\in\CC_1,\CC_2$. Suppose one of the following is true
\begin{enumerate}
\item $\langle(\CC_1)_0,u_{k_1}\rangle=\langle(\CC_2)_0,u_{k_2}\rangle$ for $k_1,k_2\in 2\N_0$ or
\item $\singletonw\otimes\singletonb\in\CC_1,\CC_2$ and $\langle(\CC_1)_0,s_{k_1}\rangle=\langle(\CC_2)_0,s_{k_2}\rangle$ for some $k_1,k_2\in\N_0$.
\end{enumerate}
Then $k:=k_1=k_2$ and
\begin{itemize}
	\item[(a)] $\CC_1=\CC_2$ if $k$ is even,
	\item[(b)] $\langle\CC_1,\singletonw\rangle=\langle\CC_2,\singletonw\rangle$ if $k$ is odd.
\end{itemize}
\end{lem}
\begin{proof}
	Suppose $\langle(\CC_1)_0,q_{k_1}\rangle=\langle(\CC_2)_0,q_{k_2}\rangle$, where $q_j$ denotes $u_j$ in case (1) or $s_j$ in case (2). Then $k_1=k_2$ follows from the fact that $k(\langle\CC_0,q_j\rangle)=j$. If $k$ is even, we can use Lemma~\ref{l.bij1} to obtain
$$\CC_1=\langle(\CC_1)_0,\paarpartww\rangle=\langle(\CC_1)_0,q_{k_1},\paarpartww\rangle=\langle(\CC_2)_0,q_{k_2},\paarpartww\rangle=\langle(\CC_2)_0,\paarpartww\rangle=\CC_2.$$
If $k$ is odd, we have
\begin{eqnarray*}
\langle\CC_1,\singletonw\rangle&=&\langle(\CC_1)_0,\paarpartww,\singletonw\rangle=\langle(\CC_1)_0,s_{k_1},\paarpartww\rangle=\cr
&=&\langle(\CC_2)_0,s_{k_2},\paarpartww\rangle=\langle(\CC_2)_0,\paarpartww,\singletonw\rangle=\langle\CC_2,\singletonw\rangle.
\end{eqnarray*}
\end{proof}


\subsection{The classification theorem}

\begin{thm}[Theorem \ref{ThA}]\label{Th}
	Every globally colorized category $\CC$ is determined by the number $k(\CC)$ and a non-colored category of partitions not containing the singleton. Therefore, the right column of Table \ref{t} forms a complete classification of globally colorized categories. All of them are pairwise inequivalent except for the rows denoted by asterisk $(*)$, which are special instances of the last family parametrized by normal subgroups of $\Z_2^{*\infty}$.
\end{thm}

\afterpage{
\clearpage
\begin{landscape}
\begin{eqnarray*}
\langle\rangle &\longrightarrow& \OOO_{\rm glob}(k)=\langle u_k,\paarpartww\otimes\paarpartbb\rangle,\;k\in 2\N_0\cr
	\langle \vierpart\rangle &\longrightarrow& \HHH_{\rm glob}(k)=\langle u_k, \vierpartwbwb,\paarpartww\otimes\paarpartbb\rangle,\; k\in 2\N_0\cr
\langle \vierpart,\singleton\otimes\singleton\rangle &\longrightarrow& \SSS_{\rm glob}(k)=\langle s_k, \vierpartwbwb,\singletonw\otimes\singletonb,\paarpartww\otimes\paarpartbb\rangle,\; k\in \N_0\cr
\langle \singleton\otimes\singleton\rangle &\longrightarrow& \BBB_{\rm glob}(k)=\langle s_k, \singletonw\otimes\singletonb,\paarpartww\otimes\paarpartbb\rangle,\; k\in 2\N_0\cr
\langle \legpart\rangle &\longrightarrow& \BBB'_{\rm glob}(k)=\langle s_k, \positionerwwbb,\paarpartww\otimes\paarpartbb\rangle,\; k\in \N_0\cr
\langle \crosspartrot\rangle &\longrightarrow& \OOO_{\rm grp,glob}(k)=\langle u_k, \crosspartrotwwbb,\paarpartww\otimes\paarpartbb\rangle,\;k\in 2\N_0\cr
	\langle \vierpart,\crosspartrot\rangle &\longrightarrow& \HHH_{\rm grp,glob}(k)=\langle u_k, \vierpartwbwb,\crosspartrotwwbb,\paarpartww\otimes\paarpartbb\rangle,\; k\in 2\N_0\quad (*)\cr
\langle \vierpart,\singleton\otimes\singleton,\crosspartrot\rangle &\longrightarrow& \SSS_{\rm grp,glob}(k)=\langle s_k, \vierpartwbwb,\singletonw\otimes\singletonb,\crosspartrotwwbb,\paarpartww\otimes\paarpartbb\rangle,\; k\in \N_0\cr
\langle \singleton\otimes\singleton,\crosspartrot\rangle &\longrightarrow& \BBB_{\rm grp,glob}(k)=\langle s_k, \singletonw\otimes\singletonb,\crosspartrotwwbb,\paarpartww\otimes\paarpartbb\rangle,\; k\in\N_0\cr
\langle \halflibpartrot\rangle &\longrightarrow& \OOO_{\rm hl,glob}(k)=\langle u_k, \halflibpartrotwwwbbb,\paarpartww\otimes\paarpartbb\rangle,\;k\in 2\N_0\cr
	\langle \vierpart,\halflibpartrot\rangle &\longrightarrow& \HHH_{\rm hl,glob}(k,0)=\langle u_k, \vierpartwbwb,\halflibpartrotwwwbbb,\paarpartww\otimes\paarpartbb\rangle,\; k\in 2\N_0\quad (*)\cr
	\langle \vierpart,\halflibpartrot,h_s\rangle &\longrightarrow& \HHH_{\rm hl,glob}(k,s)=\langle u_k,h_s, \vierpartwbwb,\halflibpartrotwwwbbb,\paarpartww\otimes\paarpartbb\rangle,k\in 2\N_0, s\ge 3, \quad (*)\cr
\langle \singleton\otimes\singleton,\halflibpartrot\rangle &\longrightarrow& \BBB_{\rm hl,glob}(k)=\langle s_k, \singletonw\otimes\singletonb,\halflibpartrotwwwbbb,\paarpartww\otimes\paarpartbb\rangle,\; k\in 2\N_0\cr
\langle \pi_s\rangle &\longrightarrow& \HHH_\pi(k,s)=\langle u_k,\pi_s,\paarpartww\otimes\paarpartbb\rangle,\;k\in 2\N_0,s\ge 2\cr
\langle \pi_l\mid l\in\N\rangle &\longrightarrow& \HHH_\pi(k,\infty)=\langle u_k,\pi_l,\paarpartww\otimes\paarpartbb\mid l\in\N\rangle,\;k\in 2\N_0\cr
A\unlhd\Z_2^{*\infty} \mathrm{sS}_\infty \hbox{inv.} &\longrightarrow& \HHH_A(k)=\langle u_k,A_0,\paarpartww\otimes\paarpartbb\rangle,\;k\in 2\N_0\cr
\singleton\otimes\singleton\not\in A
\end{eqnarray*}
	\captionof{table}{Complete classification of globally colorized categories}
	\label{t}
\end{landscape}
}

Before proving the theorem, let us comment on Table 1. Most of its notation is explained in Subsections \ref{ssec.important} and \ref{ssec.classnocol}.

The first column of the table lists all categories of non-colored partitions that do not contain the singleton~$\singleton$. The summary of the full classification is provided in Subsection \ref{ssec.classnocol}. See also \cite{raum2013full}.

Note that there are two instances of group-theoretical categories that contain $\singleton\otimes\singleton$ and hence they have to be treated separately. Namely, it is $\langle\vierpart,\singleton,\crosspartrot\rangle$ (which should not be considered at all since it contains the singleton) and $\langle\vierpart,\singleton\otimes\singleton,\crosspartrot\rangle$.

Besides that, we also treat separately those group-theoretical categories that contain the crossing partition $\crosspartrot$ or the half-liberating partition $\halflibpartrot$. The corresponding rows are denoted by asterisk $(*)$.

The right column lists the corresponding globally colorized categories. We use the following notation. By $h_s^0$ we denote a colored counterpart of $h_s$ for which $c(h_s^0)=0$. Thanks to the global colorization, it does not matter which particular colorization we choose. For definiteness, we can say, for example, that the colors in $h_s^0$ alternate beginning with white. In the same way we define $\pi_s^0$. For $A$ a $\mathrm{sS}_\infty$-invariant normal subgroup of $\Z_2^{*\infty}$, we denote by $A_0$ the category of partitions arising from the non-colored category of partitions corresponding to $A$ in the sense of Definition \ref{d.0}.

Note that the first five rows describe the classification of globally colorized categories of non-crossing partitions, the following four lines classify the globally colorized group categories, i.e.\ those containing the crossing partition $\crosspartrotwwbb$ (cf.\ results obtained in \cite{tarragowebercombina}), and the following four rows classify all globally colorized half-liberated categories of partitions, i.e. those containing the half-liberating partition \halflibpartrotwwwbbb, but not the crossing partition \crosspartrotwwbb.

\begin{proof}[Proof of Theorem \ref{Th}]
	Proposition \ref{p.bij} tells us how to obtain all globally colorized categories with $k=0$ from non-colored ones (containing partitions of even length). Lemma \ref{l} tells us how to obtain all globally colorized categories from those with $k=0$.

	The left column of Table \ref{t} contain all non-colored categories of partitions of even length. All of them are pairwise distinct (except for the ones denoted by asterisk being special instances of the last one). In the right column, we construct the corresponding globally colorized categories. For $k=0$, the generators of the category are given by Lemma \ref{l.bij2}. For $k\neq0$, we just have to add the partition $u_k$ or $s_k$ according to Lemma \ref{l}.

	To prove that all categories we have constructed are pairwise distinct, we use Lemma \ref{l.eq}. According to this lemma, distinct non-colored categories $\CC_1\neq\CC_2$ can lead to equal globally colorized categories only in the case $\singleton\otimes\singleton\in\CC_1,\CC_2$ if $\langle\CC_1,\singletonw\rangle=\langle\CC_2,\singletonw\rangle$. The only pairs $(\CC_1$, $\CC_2)$ of non-colored categories containing $\singleton\otimes\singleton$ and satisfying $\langle\CC_1,\singleton\rangle=\langle\CC_2,\singleton\rangle$ are $(\langle \singleton\otimes\singleton\rangle,\langle \legpart\rangle)$ and $(\langle \singleton\otimes\singleton,\crosspartrot\rangle,\langle \singleton\otimes\singleton,\halflibpartrot\rangle)$. For the corresponding globally colorized categories we may easily see that indeed $\BBB_{\rm glob}(k)=\BBB'_{\rm glob}(k)$ and $\BBB_{\rm grp,glob}(k)=\BBB_{\rm hl,glob}(k)$ for $k$ odd.
\end{proof}

\begin{rem} The set of generators for the categories is of course not unique. We used the partition $s_k$ for the cases $\singletonw\otimes\singletonb\in\CC$ and the partition $u_k$ for the rest. Nevertheless, as we see from Lemma \ref{l}, we could have used $u_k$ for all cases when $\CC$ consists of partitions of even length and $s_k$ for the rest. In addition, we could reformulate Lemmata \ref{l} and \ref{l.eq} and use the block partition $b_k$ in all cases when $\vierpartwbwb\in\CC$.
\end{rem}

\begin{rem}
Our results in the non-crossing case and the group case indeed match with those obtained in \cite{tarragowebercombina}. The only difference is that, for the hyperoctahedral categories, the generator $b_k$ is used instead of $u_k$ in \cite{tarragowebercombina}.
\end{rem}

\section{Unitary easy quantum groups}\label{sec.unitary}
In this section we summarize the definition of the quantum group corresponding to a category of partitions. This correspondence was described in \cite{tarragoweberopalg}. To every category of partitions a unitary compact quantum group is assigned. It generalizes the approach invented by Banica and Speicher in \cite{banica2009liberation}, which assigns orthogonal quantum groups to categories of non-colored partitions.

\subsection{Compact matrix quantum groups}The following definitions were first formulated by Woronowicz in \cite{woronowicz1987compact}. For more details about this subject, see e.g. \cite{Timmermann, Neshveyev}.

\begin{defn}
Let $A$ be a C*-algebra, $u_{ij}\in A$, where $i,j=1,\dots,n$ for some $n\in\N$. Denote $u:=(u_{ij})_{i,j=1}^n\in M_n(A)$. The pair $(A,u)$ is called a \emph{compact matrix quantum group} if
\begin{enumerate}
\item the elements $u_{ij}$ $i,j=1,\dots, n$ generate $A$,
\item the matrices $u$ and $u^t=(u_{ji})$ are invertible,
\item the map $\Delta\colon A\to A\otimes_{\rm min} A$ defined as $\Delta(u_{ij}):=\sum_{k=1}^n u_{ik}\otimes u_{kj}$ is a $*$-homomorphism.
\end{enumerate}
\end{defn}

Compact matrix quantum groups are generalizations of compact matrix groups in the following sense. For $G\subseteq M_n(\C)$ we can take the algebra of continuous functions $A:=C(G)$. This algebra is generated by the functions $u_{ij}\in C(G)$ assigning to each matrix $g\in G$ its $(i,j)$-th element $g_{ij}$. The so called \emph{co-multiplication} $\Delta\colon C(G)\to C(G)\otimes C(G)\simeq C(G\times G)$ is connected with the matrix multiplication on $G$ by $\Delta(f)(g,h)=f(gh)$ for $f\in C(G)$ and $g,h\in G$.

Therefore, for a general compact matrix quantum group $G=(A,u)$, the algebra $A$ should be seen as an algebra of non-commutative functions defined on some non-commutative compact underlying space. For this reason, we often denote $A=C(G)$ even if $A$ is not commutative. The matrix $u$ is called the \emph{fundamental representation} of $G$.

\begin{defn}
A compact matrix quantum group $H=(C(H),u^H)$ is a \emph{quantum subgroup} of $G=(C(G),u^G)$, denoted as $H\subseteq G$, if $u^G$ and $u^H$ have the same size and there is a $*$-homomorphism $\phi\colon C(G)\to C(H)$ sending $u^G_{ij}\mapsto u^H_{ij}$.
\end{defn}

\begin{defn}\label{d.QGid}
We say that compact matrix quantum groups $G=(C(G),u)$ and $H=(C(H),v)$, where $u$ and $v$ are of the same size, are \emph{identical} if there is a $*$-isomorphism $\phi\colon C(G)\to C(H)$ mapping $u_{ij}\mapsto v_{ij}$.
\end{defn}

There are two particularly important examples of compact matrix quantum groups. It is the \emph{free unitary quantum group} $U_n^+$ and the \emph{free orthogonal quantum group} $O_n^+$, which were defined by Wang in \cite{wang1995free} as follows.
$$C(U_n^+)=C^*(u_{ij}\mid \hbox{$u$ and $u^t$ are unitary}),$$
$$C(O_n^+)=C^*(u_{ij}\mid \hbox{$u$ is orthogonal}).$$
By $u$ being unitary we mean simply $uu^*=u^*u=1$, where $u^*=(u^*_{ji})$. By $u$ being orthogonal, we mean $u_{ij}=u_{ij}^*$ and $uu^t=u^tu=1$, where $u^t=(u_{ji})$. A quantum group $G$ is called \emph{unitary} or \emph{orthogonal} if it is a subgroup of $U_n^+$ or $O_n^+$, i.e. if its fundamental representation satisfies the relations of $U_n^+$ or $O_n^+$.

Another example of a compact matrix quantum group is the group C*-algebra $C^*(G)$ corresponding to some discrete group $G$. We will, in particular, use the quantum group $\hat\Z_k=(C^*(\Z_k),(z))$ and $\hat\Z=(C^*(\Z),(z))$, where
$$C^*(\Z_k)=(z\mid z^*z=zz^*=1, z^k=1),\quad C^*(\Z)=(z\mid z^*z=zz^*=1).$$

\subsection{Relations associated to partitions} Each partition defines a C*-algebraic relation, which in the end gives rise to a C*-algebra of a quantum group corresponding to some category of partitions.

\begin{defn}
Let $p\in P^\twocol(k)$ be a partition, $\alpha=(\alpha_1,\dots,\alpha_k)$ a multi-index with $\alpha_i\in\{1,\dots,n\}$. The multi-index can also be understood as a function assigning to each point in $p$ a number from 1 to $n$. We define $\delta_p(\alpha):=1$ if $\alpha$ is constant on the blocks, otherwise we put $\delta_p(\alpha)=0$.
\end{defn}

We can give the following example.
$$
\symstrut 0.7cm 
\hbox{For}\quad
p=
\begin{picture}(7,1)(0,0.5)
\put(-1,0){\uppartiii{2}{1}{2}{7}}
\put(-1,0){\uppartii{1}{3}{4}}
\put(-1,0){\uppartii{1}{5}{6}}
\put(0.05,-0.35){$\bullet$}
\put(1.05,-0.35){$\circ$}
\put(2.05,-0.35){$\bullet$}
\put(3.05,-0.35){$\bullet$}
\put(4.05,-0.35){$\bullet$}
\put(5.05,-0.35){$\circ$}
\put(6.05,-0.35){$\bullet$}
\end{picture}
\quad
\hbox{we have}
\quad
\begin{matrix}
\delta_p(1,1,3,3,6,6,1)=1,\\
\delta_p(1,1,3,3,1,1,1)=1,\\
\delta_p(1,2,3,3,6,6,1)=0.\\
\end{matrix}
$$

\begin{defn}
Let $n\in\N$, $p\in P^\twocol(k)$ be a partition an let $s_1,\dots,s_k\in\{\circ,\bullet\}$ be its color pattern. The following equation
$$\delta_p(\beta)=\sum_{\gamma_1,\dots,\gamma_k=1}^n \delta_p(\gamma)u_{\beta_1\gamma_1}^{s_1}\cdots u_{\beta_k\gamma_k}^{s_k},$$
where we set $u_{ij}^\circ=u_{ij}$ and $u_{ij}^\bullet=u_{ij}^*$, will be called the \emph{relation corresponding to the partition $p$} and denoted $R_p(u)$.
\end{defn}

\subsection{Important examples of relations}\label{ssec.relations}

We list a few partitions and the relations that can be derived directly from the definition. We also mention an equivalent formulation of those relations assuming that the matrices $u=(u_{ij})$ and $u^t=(u_{ji})$ are unitary, i.e.
$$\sum_k u_{ik}u^*_{jk}=\sum_k u^*_{ki}u_{kj}=\sum_k u^*_{ik}u_{jk}=\sum_k u_{ki}u^*_{kj}=\delta_{ij}.$$
See \cite{tarragoweberopalg}, Section 2 for more details.

\begin{eqnarray*}
\singletonw\otimes\singletonb\colon&\quad& 1=\left(\sum_{j_1}u_{i_1j_1}\right)\left(\sum_{j_2}u_{i_2j_2}\right)\\
&\Leftrightarrow& \sum_k u_{kj}=\sum_l u_{il}\\
\positionerwwbb\colon&\quad& \delta_{i_2i_4}=\sum_{j_1,j_2,j_3}u_{i_1j_1}u_{i_2j_2}u^*_{i_3j_3}u^*_{i_4j_2}\\
&\Leftrightarrow& u_{ij}\left(\sum_{k_1}u_{k_1j_1}\right)=\left(\sum_{l_1}u_{i_1l_1}\right)u_{ij}\\
\paarpartww\otimes\paarpartbb\colon&\quad& \delta_{i_1k_1}\delta_{i_2k_2}=\left(\sum_{j_1}u_{i_1j_1}u_{k_1j_1}\right)\left(\sum_{j_2}u^*_{i_2j_2}u^*_{k_2j_2}\right)\\
&\Leftrightarrow& u^*_{ij}u_{kl}=u_{ij}u^*_{kl}\\
u_{2l}\colon&\quad&\delta_{i_1k_1}\cdots\delta_{i_lk_l}=\left(\sum_{j_1}u_{i_1j_1}u_{k_1j_1}\right)\cdots \left(\sum_{j_l}u_{i_lj_l}u_{k_lj_l}\right)\\
&\Leftrightarrow& u_{i_1j_1}\cdots u_{i_lj_l}=u^*_{i_1j_1}\cdots u^*_{i_lj_l}\\
s_k\colon&\quad& 1=\left(\sum_{j_1}u_{i_1j_1}\right)\cdots\left(\sum_{j_k}u_{i_kj_k}\right)
\end{eqnarray*}

\newpage 

\subsection{Definition of unitary easy quantum groups}

\begin{defn} A compact matrix quantum group $G$ is called \emph{easy} if there is a set of partitions $S\subseteq P^\twocol$ such that $G=(A,u)$, where
$$A=(u_{ij},\;i,j=1,\dots,n\mid R_p(u), p\in S;\;\hbox{$u$ and $u^t$ unitary})$$
for some $n\in\N$.
\end{defn}

It can be shown that for every set $S\subseteq P^\twocol$ and for every $n$ there exists such an easy quantum group. Moreover, we have the following.

\begin{lem}[{\cite[Lemma 2.7, see also Section 4.1.4]{weber2017alg}}] The easy quantum group corresponding to a set $S\subseteq P^\twocol$ is identical to the easy quantum group corresponding to the category $\langle S\rangle$ generated by this set.
\end{lem}

\begin{thm}[{\cite[Corollary 3.13]{tarragoweberopalg}}]
Categories of partitions are in one-to-one correspondence with sequences $(G_n)_{n=1}^\infty$ of easy quantum groups with fundamental representation of size $n$.
\end{thm}

\section{Quantum groups corresponding to globally colorized categories}\label{sec.globqg}
In \cite{tarragoweberopalg} the authors study several possibilities for constructing unitary quantum groups from the orthogonal ones. They extend the work of Wang who defined tensor and free products of quantum groups in \cite{wang1995free,wang1995tensor}. Using these techniques they construct the quantum groups corresponding to categories of colored partitions found in \cite{tarragowebercombina}. In this section, we follow their argumentation to construct the quantum groups of the globally colorized categories classified in this article. 

\subsection{The tensor complexification}

\begin{prop}[\cite{wang1995tensor}] Let $G=(C(G),u)$ and $H=(C(H),v)$ be compact matrix quantum groups. Then $G\times H:=(C(G)\otimes_{\rm max}C(H),u\oplus v)$ is a compact matrix quantum group. For the co-multiplication we have that
$$\Delta_\times(u_{ij}\otimes 1)=\Delta_G(u_{ij}),\quad\Delta_\times(1\otimes v_{kl})=\Delta_H(v_{kl}).$$
\end{prop}

Note that the algebra $C(G)\otimes_{\rm max}C(H)$ can be described as a universal C*-algebra generated by elements $u_{ij}$ and $v_{kl}$ such that every $u_{ij}$ commutes with every $v_{kl}$, the elements $u_{ij}$ satisfy the same relations as $u_{ij}\in C(G)$ and the elements $v_{kl}$ satisfy the same relations as $v_{kl}\in C(H)$. Thus, the matrix 
$$u\oplus v=\begin{pmatrix}u&0\\0&v\end{pmatrix}\in M_n(C(G)\otimes_{\rm max}C(H))$$
indeed consists of generators of the algebra $C(G)\otimes_{\rm max}C(H)$. From now on, we will use such interpretation of the maximal tensor product and we will write just $u_{ij}v_{kl}$ instead of $u_{ij}\otimes v_{kl}$ without the explicit tensor sign.

\begin{prop}[\cite{tarragoweberopalg}] Let $G=(C(G),u)$ and $H=(C(H),v)$ be compact matrix quantum groups. Let $A$ be the C*-subalgebra of $C(G)\otimes_{\rm max} C(H)$ generated by the products $u_{ij}v_{kl}$, i.e. generated by the elements of the matrix $u\otimes v$. Then $G\tensorglued H:=(A,u\otimes v)$ is a compact matrix quantum group. For the co-multiplication we have that
$$\Delta_{\tensorglued}(u_{ij}v_{kl})=\Delta_G(u_{ij})\Delta_H(v_{kl})$$
\end{prop}

Note that the co-multiplication $\Delta_{\tensorglued}$ is just the restriction of the comultiplication $\Delta_\times$ to the subalgebra $A$. The quantum group $G\tensorglued H$ is called the \emph{glued tensor product} of $G$ and $H$.

In this work, we will use the product of the form $G\tensorglued\hat\Z_k$ or $G\tensorglued\hat\Z$, which is called the \emph{tensor $k$-complexification} or \emph{tensor complexification}, respectively. Note that if $u$ is the fundamental representation of $G$ and $z$ is the generator of $\hat\Z_k$, then the fundamental representation of the tensor $k$-complexification is the matrix $uz=(u_{ij}z)$, which is of the same size as $u$.

\subsection{Quantum groups corresponding to globally colorized categories} The following theorem was proven in \cite{tarragoweberopalg} for the non-crossing and group categories of partitions. Here, we apply the same technique of proof on all of the globally colorized categories.

\begin{thm}[Theorem \ref{ThB}]
\label{Th2}
Let $\CC$ be a globally colorized category of partitions, denote $k:=k(\CC)$. Denote by $H\subset O_n^+$ the quantum group corresponding to the non-colored category of partitions $\langle\CC,\paarpartww\rangle$. Then $\CC$ corresponds to the quantum group $H\tensorglued\hat\Z_k\subset U_n^+$. In the case $k=0$, we replace $\hat\Z_k$ by $\hat\Z$.

\end{thm}
\begin{proof}
We will divide the proof into two cases: Case (a) $\positionerwwbb\in\CC$ and case (b) $\positionerwwbb\not\in\CC$. According to Lemma \ref{l}, we have $\CC=\langle\CC_0,s_k\rangle$ in case (a) and $\CC=\langle\CC_0,u_k\rangle$ in case (b). Since large parts of the proof will be identical for both cases, let us denote $\CC=\langle\CC_0,q_k\rangle$, where $q_k$ is either $s_k$ or $u_k$. Denote by $G$ the quantum group associated to the category $\CC$.

Denote by $\Rel(u)$ the relations associated to the category $\CC_0$. Note that the uni-colored pair $\paarpartww=u_1$ corresponds to the relation making the matrix $u$ orthogonal. So, we have
$$C(G)=(u_{ij}\mid \Rel(u), R_{q_k}(u), \hbox{$u$ and $u^t$ unitary}),$$
$$C(H)=(v_{ij}\mid \Rel(v), R_{q_k}(v), \hbox{$v$ orthogonal}),$$
$$C(H)\otimes_{\rm max}C^*(\Z_k)=\left(v_{ij},z\mathrel{\bigg|}\begin{matrix}\Rel(v),R_{q_k}(v), \hbox{$v$ orthogonal},\\ v_{ij}z=zv_{ij},zz^*=z^*z=1, z^k=1\end{matrix}\right),$$
$$C(H\tensorglued\hat\Z_k)\subseteq C(H)\otimes_{\rm max}C^*(\Z_k)\quad\hbox{generated by $u'_{ij}:=v_{ij}z$}.$$
	According to Definition \ref{d.QGid}, we have to show that there exists a $*$-isomorphism $\alpha\colon C(G)\to C(H\tensorglued\hat\Z_k)$ mapping $u_{ij}\mapsto u_{ij}'$. We will show it in two steps. First, we are going to use the universal property of $C(G)$ to find such a surjective map $\alpha$. Secondly, we find a $*$-homomorphism $\beta\colon C(H)\otimes_{\rm max}C^*(\Z_k)\to M_t(C(G))$ such that $\beta\circ\alpha=\iota$, where $t=1$ in case (a) and $t=2$ in case (b) and $\iota\colon C(G)\to M_t(C(G))$ is the embedding $x\mapsto x\cdot 1_{\C^t}$. This will prove the injectivity of $\alpha$.

\begin{center}
\includegraphics{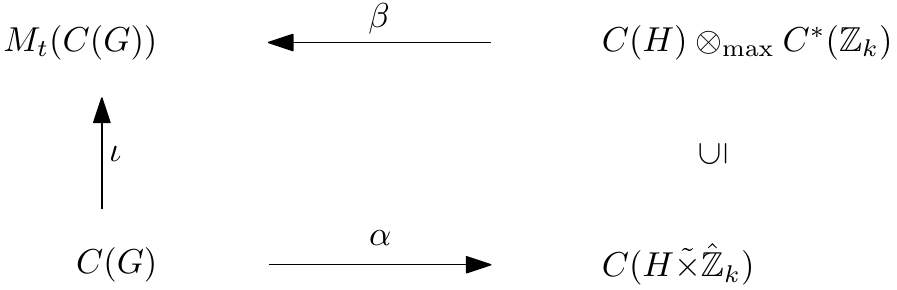}
\end{center}

\smallskip\noindent\textbf{Step 1.}\space\emph{There is a surjective $*$-homomorphism $\alpha\colon C(G)\to C(H\tensorglued\hat\Z_k)$ mapping $u_{ij}\mapsto u'_{ij}$.}
\smallskip

We show that the elements $u'_{ij}\in C(H\tensorglued\hat\Z_k)\subseteq C(H)\otimes_{\rm max} C^*(\Z_k)$ satisfy the relations of $u_{ij}$ in $C(G)$.

Indeed, the unitarity is clear since $u'u'^*=vzz^*v^*=vv^*=1$ and similarly for $u^*u$ and for $u^t$.

Moreover, all elements $p\in\CC_0$ satisfy $c(p)=0$, so they have the same amount of white and black points. Thus, in the relations $\Rel(u)$ there is the same amount of conjugated $u^*_{ij}$ as not conjugated $u_{ij}$. If we put $vz$ instead of $u$, there is the same amount of $z$ and $z^*$ in every relation. So, since $z$ and $z^*$ commute with everything, they all cancel out. Therefore, the relations $\Rel(vz)=\Rel(v)$ are also satisfied in $C(H)\otimes_{\rm max} C^*(\Z_k)$.

Similarly for the case $R_{q_k}$. We know that $c(q_k)=k$ and since $z^k=1$, it follows that $R_{q_k}(vz)=R_{q_k}(v)$, which is satisfied in $C(H)\otimes_{\rm max} C^*(\Z_k)$.

Finally, from the universal property of $C(G)$, we know that there is a $*$-homomorphism $\alpha\colon C(G)\to C(H\tensorglued\hat\Z_k)$ mapping $u_{ij}\mapsto u'_{ij}$. Since $C(H\tensorglued\hat\Z_k)$ is generated by $u'_{ij}$, the homomorphism must be surjective.

\smallskip\noindent\textbf{Step 2a.}\space\emph{Suppose $\positionerwwbb\in\CC$. Define in $C(G)$
$$z':=\sum_l u_{il}=\sum_l u_{lj},\quad v'_{ij}:=u_{ij}z'^*.$$
There is a $*$-homomorphism $\beta\colon C(H)\otimes_{\rm max}C^*(\Z_k)\to C(G)$ mapping $v_{ij}\mapsto v'_{ij}$, $z\mapsto z'$ satisfying $\beta\circ\alpha=\id$.}
\smallskip

We will again show that $v'_{ij}$ and $z'$ in $C(G)$ satisfy the relations of $v_{ij}$ and $z$ in $C(H)\otimes_{\rm max}C^*(\Z_k)$. Then the existence of $\beta$ will follow from the universal property of $C(H)\otimes_{\rm max}C^*(\Z_k)$.

First of all note that $\positionerwwbb\in\CC$ implies that $\singletonw\otimes\singletonb\in\CC$, so we indeed have $\sum_l u_{il}=\sum_l u_{lj}$ for any $i,j$ (see Subsection \ref{ssec.relations}). Since we are in case (a), we have $s_k\in\CC$ if $k\neq 0$, which corresponds to the relation $z'^k=1$. The unitarity of $z'$ follows from
$$z'z'^*=\sum_j u_{ij}\sum_l u^*_{il}=\sum_j\delta_{jl}=1=\sum_ju^*_{ij}\sum_lu_{il}=z'^*z'.$$

The relation for $\positionerwwbb$ means that $u_{ij}z'=z'u_{ij}$, which implies that
$$v'_{ij}z'=u_{ij}z'^*z'=u_{ij}=z'u_{ij}z'^*=z'v'_{ij}.$$

The relation for $\paarpartww\otimes\paarpartbb$ implies that
$$u_{ij}z'^*=\sum_l u_{ij}u^*_{kl}=\sum_l u^*_{ij}u_{kl}=u^*_{ij}z',$$
from which we can deduce that
$$v'_{ij}=u_{ij}z'^*=u^*_{ij}z'=z'u^*_{ij}=v'^*_{ij}.$$

The orthogonality of $v'$ follows simply from the equality $v'v'^*=uu^*$ and $v'^*v'=u^*u$, which follows from the fact that $u$ commutes with $z'$. Similarly, all relations $\Rel(v')$ are satisfied since they are equivalent to $\Rel(u)$ thanks to the fact that $u$ commutes with $z'$ and hence all occurrences of $z'$ cancel with $z'^*$. Finally, if $k\neq0$, we can use the relation $z'^k=1$ to show that the relation $R_{s_k}(v')$ is equivalent to $R_{s_k}(u)$.

\smallskip\noindent\textbf{Step 2b.}\space\emph{Suppose $\positionerwwbb\not\in\CC$. Denote $w:=\sum_l u_{il}^2=\sum_l u_{lj}^2\in C(G)$. Define in $M_2(C(G))$
$$z':=\begin{pmatrix}0&w\\1&0\end{pmatrix},\quad v'_{ij}:=\begin{pmatrix}0&u_{ij}\\u^*_{ij}&0\end{pmatrix}.$$
There is a $*$-homomorphism $\beta\colon C(H)\otimes_{\rm max}C^*(\Z_k)\to M_2(C(G))$ mapping $v_{ij}\mapsto v'_{ij}$, $z\mapsto z'$ satisfying $\beta\circ\alpha=\iota$.}
\smallskip

Again, we will prove that $v'_{ij}$ and $z'$ satisfy the appropriate relations and the statement will follow from the universal property of the algebra $C(H)\otimes_{\rm max}C^*(\Z_k)$.

From the relation corresponding to $\paarpartww\otimes\paarpartbb$ it follows that indeed
$$\sum_l u_{il}^2=\sum_{l,m}u_{il}u_{il}u^*_{mj}u_{mj}=\sum_{l,m}u_{il}u^*_{il}u_{mj}u_{mj}=\sum_mu_{mj}^2,$$
so $w$ is well-defined. We can also use this relation to prove that
$$ww^*=\sum_{l_1}u_{il_1}^2\sum_{l_2}u_{il_2}^{*2}=\sum_{l_1}u_{il_1}u^*_{il_1}\sum_{l_2}u_{il_2}u^*_{il_2}=1$$
and similarly $w^*w=1$, so $w$ is unitary. Then we use the relation to prove that
$$u^*_{ij}w=\sum_lu^*_{ij}u_{il}^2=\sum_l u_{ij}u^*_{il}u_{il}=u_{ij}$$
and similarly $wu^*_{ij}=u_{ij}$. And also together with the relation for $u_k$ to prove that
$$w^{k/2}=\sum_{l_1}u_{il_1}u_{il_1}\cdots\sum_{l_{k/2}}u_{il_{k/2}}u_{il_{k/2}}=\sum_{l_1}u_{il_1}u^*_{il_1}\cdots\sum_{l_{k/2}}u_{il_{k/2}}u^*_{il_{k/2}}=1.$$

One can now easily check the validity of all the relations. In particular, we have that
$$v'_{ij}z'=z'v'_{ij}=u_{ij}\begin{pmatrix}1&0\\0&1\end{pmatrix},$$
so $v'_{ij}=u_{ij}z'^*\,1_{\C^2}$. Therefore, for the relations $\Rel(v)$ we can apply the same argument as in case (a).
\end{proof}

\subsection{Characterization of the tensor complexification}

In Theorem \ref{Th2}, we constructed precisely all quantum groups of the form $H\tensorglued\hat\Z_k$, where either
\begin{enumerate}
\item $H\subset O_n^+$ corresponds to a category $\tilde\CC$ such that $\singleton\not\in\tilde\CC$ and $k$ is even or
\item $H\subset O_n^+$ corresponds to a category $\tilde\CC$ such that $\singleton\in\tilde\CC$ and $k$ is odd.
\end{enumerate}
Indeed, consider a globally colorized category $\CC$. If $k:=k(\CC)$ is odd, then $\CC$ must contain a partition of odd length. According to Lemmata \ref{l.even} and \ref{l}, we have that $s_k\in\CC$ and hence $\singletonw\in\bar\CC:=\langle\CC,\paarpartww\rangle$. On the other hand, if $k$ is even, then $\CC$ does not contain any partition of odd length and hence $\singletonw\not\in\bar\CC$.

Conversely, let $\tilde\CC$ be any category of non-colored partitions corresponding to an orthogonal quantum group $H\subset O_n^+$ and denote $\bar\CC:=\Psi^{-1}(\tilde\CC)$ the corresponding category in terms of two-colored partitions. If $\singleton\in\tilde\CC$, then we can construct $\CC:=\langle\bar\CC_0,s_k\rangle$ for $k$ odd and convince ourselves that $\bar\CC=\langle\CC,\paarpartww\rangle$. Thus, Theorem \ref{Th2} implies that $H\tensorglued\hat\Z_k$ is a unitary easy quantum group corresponding to the category $\CC$. If $\singleton\not\in\tilde\CC$, then we can similarly construct $\CC:=\langle\bar\CC_0,u_k\rangle$ for $k$ even and use Theorem \ref{Th2} to see that $H\tensorglued\hat\Z_k$ is the quantum group corresponding to $\CC$.

Now, we would like to characterize the product $H\tensorglued\Z_k$ for an arbitrary orthogonal easy quantum group $H$ and an arbitrary number $k\in\N_0$.

There are only four non-colored categories of partitions containing the singleton. Namely, the following:
\begin{align*}
\langle\singleton,\vierpart\rangle&\qquad&\mbox{quantum group } S_n^+,\\
\langle\singleton\rangle&\qquad&\mbox{quantum group } B_n^+,\\
\langle\singleton,\vierpart,\crosspartrot\rangle&\qquad&\mbox{(quantum) group } S_n,\\
\langle\singleton,\crosspartrot\rangle&\qquad&\mbox{(quantum) group } B_n.\\
\end{align*}

Note that they can be formed by adding the singleton to the following categories:
\begin{align*}
\langle\singleton\otimes\singleton,\vierpart\rangle&\qquad&\mbox{quantum group } S_n^{\prime +}=S_n^+\tensorglued\hat\Z_2,\\
\langle\legpart\rangle&\qquad&\mbox{quantum group } B_n'^+=B_n^+\tensorglued\hat\Z_2,\\
\langle\singleton\otimes\singleton,\vierpart,\crosspartrot\rangle&\qquad&\mbox{(quantum) group } S_n\tensorglued\hat\Z_2,\\
\langle\singleton\otimes\singleton,\crosspartrot\rangle&\qquad&\mbox{(quantum) group } B_n\tensorglued\hat\Z_2.\\
\end{align*}

Their $k$-tensor complexifications for $k$ even are described by the following proposition.

\begin{prop}\label{P.glue1}
For any $k\in 2\N_0$ we have the following
\begin{align*}
S_n^+\tensorglued\hat\Z_k&=(S_n^+\tensorglued\hat\Z_2)\tensorglued\hat\Z_k\qquad &&(\mbox{category } \SSS_{\rm glob}(k)),\\
B_n^+\tensorglued\hat\Z_k&=(B_n^+\tensorglued\Z_2)\tensorglued\hat\Z_k\qquad            &&(\mbox{category } \BBB'_{\rm glob}(k)),\\
S_n\tensorglued\hat\Z_k&=(S_n\tensorglued\hat\Z_2)\tensorglued\hat\Z_k\qquad     &&(\mbox{category } \SSS_{\rm grp,glob}(k)),\\
B_n\tensorglued\hat\Z_k&=(B_n\tensorglued\hat\Z_2)\tensorglued\hat\Z_k\qquad     &&(\mbox{category } \BBB_{\rm grp,glob}(k)).\\
\end{align*}
\end{prop}
\begin{proof}
	From Theorem \ref{Th2} it follows that, for each row, the quantum group on the right hand side corresponds to the category in parentheses. Repeating the proof of Theorem \ref{Th2} for the quantum group on the left hand side proves the equality. In particular, note that all the categories are in case (a) and actually the only thing we have to check is that $v_{ij}'\in C(G)$ satisfies the relation corresponding to the singleton $\sum_l v'_{il}=\sum_l v'_{lj}=1$ for any $i,j$. This follows from the fact that
$$\sum_lv_{il}'=\sum_lu_{il}z'^*=z'z'^*=1$$
and similarly for the other relation.
\end{proof}

\begin{rem}
Those alternative descriptions of the quantum groups were actually used in \cite[Theorem 4.13, Theorem 5.3]{tarragoweberopalg}.
\end{rem}

\begin{prop}\label{P.glue2}
Let $H$ be an orthogonal easy quantum group corresponding to a category $\CC$ such that $\singleton\not\in\CC$ and let $k$ be odd. Then $H\tensorglued\Z_k=H\tensorglued\Z_{2k}$.
\end{prop}
\begin{proof}
Denote the fundamental representation of $H$ by $v$, the generator of $C^*(\Z_{2k})$ by $z$ and the generator of $C^*(\Z_k)$ by $z'$. Then we can denote the fundamental representation of $C(H\tensorglued\hat\Z_{2k})$ by $u_{ij}:=v_{ij}z$ and the fundamental representation of $C(H\tensorglued\hat\Z_k)$ by $u_{ij}':=v_{ij}z'$. We have to find an isomorphism $C(H\tensorglued\hat\Z_{2k})\to C(H\tensorglued\hat\Z_k)$ mapping $u_{ij}\mapsto u_{ij}'$.

	Let us define a homomorphism $\alpha\colon C(H)\otimes C^*(\Z_{2k})\to C(H)\otimes C^*(\Z_k)$ by $\alpha(v_{ij})=v_{ij}$ and $\alpha(z)=z'$. The existence of such a homomorphism follows from the universal property of $C(H)\otimes C^*(\Z_{2k})$ since we have $z'^{2k}=1$.

	We would also like to define a homomorphism $\beta\colon C(H)\otimes C^*(\Z_{k})\to C(H)\otimes C^*(\Z_{2k})$ by $\beta(v_{ij})=v_{ij}z^{*k}$ and $\beta(z')=z^{k+1}$. Obviously $\beta(z')$ satisfies the relation of $z'$ since $\beta(z')^k=z^{(k+1)k}=1$. Since $\singleton\not\in\CC$, we know that $\CC$ contains only partitions with even length, so all the relations of $C(H)$ contain monomials in $v_{ij}$ of even length. Thus, $\beta(v_{ij})$ satisfy the relations of $v_{ij}$ since all the $z^{*k}$ cancel out. Therefore, such a homomorphism $\beta$ exists from a universal property of $C(H)\otimes C^*(\Z_k)$.

	Now since we have
$$\alpha(u_{ij})=\alpha(v_{ij}z)=v_{ij}z'=u_{ij}',$$
$$\beta(u_{ij}')=\beta(v_{ij}z')=v_{ij}z^{*k}z^{k+1}=v_{ij}z=u_{ij},$$
	it follows that $\alpha$ restricted to $C(H\tensorglued\Z_{2k})$ is a surjective homomorphism onto $C(H\tensorglued\Z_k)$ mapping $u_{ij}\mapsto u_{ij}'$, it has an inverse provided by the map $\beta$, and hence it is the desired isomorphism.
\end{proof}

\begin{cor}\label{C.glue}
Any quantum group of the form $H\tensorglued\hat\Z_k$, where $H$ is an orthogonal easy quantum group and $k\in\N_0$, is a unitary easy quantum group corresponding to a globally colorized category.
\end{cor}
\begin{proof}
Denote $\bar\CC$ the category corresponding to the orthogonal quantum group $H$ in terms of two-colored partitions. If $\singletonw\not\in\bar\CC$ and $k$ is even, then, as we mentioned in the beginning of this subsection, $H\tensorglued\hat\Z_k$ is a unitary easy quantum group corresponding to the category $\langle\bar\CC_0,u_k\rangle$. If $k$ is odd, then, according to Proposition \ref{P.glue2}, $H\tensorglued\hat\Z_k=H\tensorglued\hat\Z_{2k}$, so it reduces to the previous case and $H\tensorglued\hat\Z_k$ corresponds to the category $\langle\bar\CC_0,u_{2k}\rangle$.

If $\singletonw\in\bar\CC$ and $k$ is odd, then again, as we mentioned in the beginning of this subsection, $H\tensorglued\hat\Z_k$ is a unitary easy quantum group corresponding to the category $\langle\bar\CC_0,s_k\rangle$. If $k$ is even, then, according to Proposition \ref{P.glue1}, $H\tensorglued\hat\Z_k=(H\tensorglued\hat\Z_2)\tensorglued\hat\Z_k$, where $H\tensorglued\hat\Z_2$ is an orthogonal easy quantum group corresponding to a category that does not contain the singleton, so the situation reduces to the first case. In fact, we again have that $H\tensorglued\hat\Z_k$ corresponds to the category $\langle\bar\CC_0,s_k\rangle=\langle\bar\CC_0,u_k\rangle$.
\end{proof}

\bibliographystyle{alpha}
\bibliography{}

\end{document}